\documentclass[11pt]{article}
\usepackage{amsmath,amssymb,amsthm,amsfonts,mathrsfs,mathtools,physics}

\usepackage[nottoc]{tocbibind}
\usepackage{dsfont} 
\usepackage[utf8]{inputenc}
\usepackage[T1]{fontenc}
\usepackage[english]{babel}
\usepackage{palatino}
\usepackage{comment}
\usepackage{fullpage}
\usepackage{enumerate}
\usepackage[toc]{appendix}
\usepackage{esint}
\usepackage{appendix}
\usepackage{url}
\usepackage{graphics,graphicx}
\usepackage{ifthen}
\usepackage[normalem]{ulem} 
\usepackage{caption}
\usepackage{tikz}

\usepackage[backref=page]{hyperref}
 \hypersetup{
 	plainpages=false,
 	unicode=false,          
 	pdftoolbar=true,        
 	pdfmenubar=true,        
 	pdffitwindow=true,      
 	pdftitle={},    
 	pdfauthor={},     
 	pdfsubject={},   
 	pdfcreator={},   
 	pdfproducer={}, 
 	pdfkeywords={}, 
 	pdfnewwindow=true,      
 	colorlinks=true,       
 	linkcolor=blue,          
 	citecolor=blue,        
 	filecolor=blue,      
 	urlcolor=blue,          
 	urlbordercolor=0 1 1
 }
\usepackage{xcolor}
\definecolor{Blue}{rgb}{0.,0.,1}

\usepackage{setspace}
\usepackage[pass]{geometry}

\newcommand{\xC}{{\rm C}}
\newcommand{\del}{\Delta t}
\newcommand{\R}{\mathbb R}
\newcommand{\N}{\mathbb N}

\newcommand{\cH}{\mathcal{H}}

\newcommand{\cL}{\mathcal{L}}

\newcommand{\cM}{\mathcal{M}}
\newcommand{\cN}{\mathcal{N}}

\newcommand{\cT}{\mathcal{T}}

\newcommand{\G}{{\mathbb G}_{d,n}}
\newcommand{\V}{\|V\|}

\newtheorem{theo}{Theorem}[section]
\newtheorem*{theo*}{Theorem}
\newtheorem{prop}[theo]{Proposition}
\newtheorem{lemma}[theo]{Lemma}
\newtheorem{dfn}[theo]{Definition}
\newtheorem*{dfn*}{Definition}
\newtheorem{cor}[theo]{Corollary}

\newtheoremstyle{rmdotless}{}{}{\upshape}{}{\bfseries}{.}{0.5em}{}
\theoremstyle{rmdotless}
\newtheorem{remk}[theo]{Remark}

\DeclareMathOperator{\mdiv}{div}

\DeclareMathOperator*{\supp}{spt}
\DeclareMathOperator*{\dist}{dist}

\DeclareMathOperator*{\lip}{Lip}

\renewcommand{\phi}{\varphi}
\renewcommand{\epsilon}{\varepsilon}
\newcommand{\e}{\epsilon}
\renewcommand{\G}{G_{d,n}}

\numberwithin{equation}{section} 
\title{On the avoidance principle for codimension $1$ space-time Brakke flows}
\author{}
\date{}

\begin{document}
\maketitle
\begin{center}
 {\bf Abdelmouksit Sagueni}\\
Institut Camille Jordan\\
{\it sagueni@math.univ-lyon1.fr}
\end{center}

\begin{abstract}

We prove that the spacetime Brakke flow constructed by Buet et al. is non-trivial as long as the initial varifold is a union of boundaries of domains of finite perimeter. In the codimension $1$ setting, we show that, starting from a smooth boundary, the support of the mass measure of the spacetime Brakke flow coincides with the support of the classical mean curvature flow.

%

\end{abstract}
\tableofcontents

\newpage
\section{Introduction}
The Brakke flow is a generalization of the mean curvature flow in the geometric measure theoretic setting, introduced by K.Brakke in \cite{brakke}. The idea was to derive an inequality involving the mass and the velocity vector (the mean curvature) which characterizes the evolution in the smooth case to a certain extent. This inequality was then used to provide a weak definition for the evolution by the mean curvature flow for less regular initial data (rectifiable varifolds), it has since been referred to as the "Brakke inequality". For any $d,n$, $1\leq d <n$ and any $d$-rectifiable varifold $V_0$ in $\R^n$, the author proved in \cite{brakke} the existence of a family of $d$-rectifiable varifolds $\big(V(t)\big)_t$ (with $V(0)=V_0$) satisfying the Brakke inequality. The construction of such family  was done via successive pushforwards with maps of velocity equal to the approximate mean curvature vector coupled with a desingularization step (to guarantee the rectifiability of the limit flow). The main issue of the former construction is that, in no situation we do know for sure whether the flow is trivial (empty $\forall t >0$) or not, as the Brakke inequality allows the sudden vanishing of the mass. In a recent work \cite{kt},\, Kim and Tonegawa ruled out this issue in a specific case. Namely, when $V_0$ is a boundary of an open partition of $\R^n$ (grain boundaries), by means of modifying the desingularization functions to control the change of volume, the authors proved that the volumes of the limit flow of the open partition change continuously, as the Brakke flow coincides with their boundaries, it is then nontrivial.\\
%

%
\subsection{Overview on the avoidance principle}
\noindent The classical avoidance principle for the mean curvature flow says that if $\cM$ and $\cN$ are two disjoint smooth compact hypersurfaces, then their respective mean curvature flows are disjoint. This avoidance property is a consequence of the maximum principle. This principle fails in higher codimensions, for instance, take the MCF of two enlaced disjoint circles in $\R^3$ (see Figure \ref{fig:enlace}). Ilmanen (\cite[Lemma 4E]{Ilmanen3}) generalized the avoidance principle to arbitrary "set-theoretic subsolutions of mean curvature flow". He  showed in \cite[Theorem 10.5]{Ilmanen} that the support of a codimension $1$ integral Brakke flow is a set-theoretic subsolution of mean curvature flow, so the avoidance principle also applies to such a Brakke flow.\\

\noindent Inspired by the works of Ilmanen and Brakke, we prove certain avoidance and approximate avoidance principles for spacetime Brakke flows and their approximations defined in  \cite{blms1}.

\begin{figure}[h!]
\centering
\includegraphics[width=75mm]{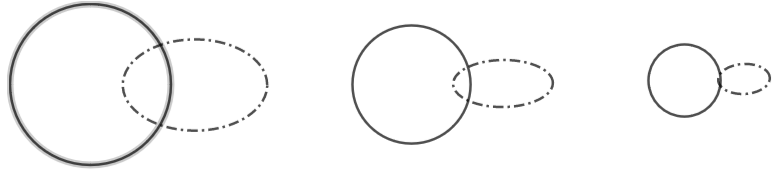}
\caption{The MCFs of two enlaced disjoint circles in $\R^3$ collide.}
\label{fig:enlace}
\end{figure}
\subsection{Our contribution}
We work in a general setting, starting from any $d$-varifold in $\R^n$, we constructed in \cite{blms1} a time-discrete approximate mean curvature flow by successive pushforwards with maps of velocity equal to the approximate mean curvature defined in \cite{kt} (after \cite{brakke}). Given that the work setting is general, we omit the desingularization step from our construction. We decorrelate the approximation scale from the time step scale, we first let the time step tend to $0$, then the approximation scale. In the first case we get a flow satisfying a Brakke equality w.r.t. its approximate mean curvature vector, we called it: the approximate mean curvature flow. In the second case, coupling the approximate mean curvature flow with the measure $dt$, we get a flow satisfying a Brakke inequality w.r.t. its spacetime mean curvature vector (more details will follow), we call it "a spacetime Brakke flow". The main contributions of this paper are:
\begin{enumerate}
\item(Nontriviality) Starting from a boundary of an open partition of $\R^n$, the spacetime Brakke flow constructed as limits of discrete approximate MCFs (when the approximation scale and the time step tend to $0$) are not trivial (Theorem \ref{thm:nontriviality}).
\item(Avoidance principle in codimension $1$) The mass measure of a spacetime Brakke flow avoids any smooth evolution by mean curvature flow (Theorem \ref{avoidance}).
\item(Support inclusion) In case we deal with a single domain with $C^2$ closed manifold, denoted by $M$. The support mass measure of a spacetime Brakke flow starting from $M$ coincides with the support of the MCF starting from $M$ as long as the latter exists (Corollary \ref{cor:avoidance}).
\end{enumerate}

\subsection{Notations and preliminaries}\label{avoidance_notations}
\begin{itemize}
\item For $k\in\N$, $\omega_k$ will denote the volume of the $k-$dimensional unit ball.
\item $B_r(x)$ and $B(r,x)$ will denote the open ball of radius $r>0$ and center $x \in \R^n$. Conventionally $B_r$ denotes the open ball of radius $r>0$ and center $0$. For closed balls we keep the same notations and replace $B$ by $\overline{B}$.
\item For any two sets $A$ and $B$ of $\R^n$ we define $A+B= \{ a+b, \, (a,b) \in  A \times B \}$.
\item For two matrices $M$ and $N$ we define the matrix scalar product by 
\begin{equation}
 M:N = \tr(MN^t)=\tr(NM^t).
\end{equation}
\item For $a < b \in \R$ and $m\geq1$, $\cT=\lbrace t_i \rbrace_{i=0}^m$ is called a subdivision of $[a,b]$ if: $a=t_0 < t_1 \dots < t_m=b$. We denote: 
\begin{equation*}
\delta(\cT) := \max\limits_i t_i - t_{i-1} \quad  i \in \lbrace 1, \dots, m \rbrace.
\end{equation*}
\item $\G$ is the Grassmannian manifold of $d$-dimensional vector subspaces of $\R^n$. We identify every element $S\in\G$ with its orthogonal projection on the $d$-subspace $S \in \cM_n(\R)$.\item For a function $u$, define the supremum norm as: $\|u\|_{\infty}=\sup\limits_x|u(x)|$.
\item $C^0, C_c^0$ denote the space of continuous functions, continuous compactly supported functions respectively.
\item $C^{1}$ denotes the space of continuously differentiable functions. We define for $\phi \in C^{1}$, $\|\phi\|_{C^1} := \|\phi\|_{\infty}+ \|D\phi\|_{\infty}$.
\item $C^{2}$ denotes the space of continuously differentiable functions. We define for $\phi \in C^{2}$, $\|\phi\|_{C^2} := \|\phi\|_{\infty}+ \|D\phi\|_{\infty}+ \| D^2 \phi \|_{\infty}$.
\item The  functions and vectors involved may depend on the space and the time, we conventionally use $\nabla$ for space derivation and $\partial_t$ for time derivation.
\item We introduce as in \cite{kt} families of test functions and vectors, $\forall j \in \N$ : 
\begin{equation}
\mathcal{A}_j := \lbrace \phi \in C^2(\R^n;\R^+) : \phi(x) \leq  1 , |\nabla\phi(x)|\leq j \phi(x) , \|\nabla^2\phi(x)\|\leq j \phi(x) \,  \text{for all} \,  x \in \R^n \rbrace,
\end{equation}
\begin{equation}
\mathcal{B}_j := \lbrace g \in C^2(\R^n;\R^n) : |g(x)|\leq j , \|\nabla g(x)\|\leq j, \|\nabla^2 g(x)\|\leq j \,  \mbox{for all} \,  x \in\R^n \, \mbox{and} \, \|g\|_{L^2(\R^n)}\leq j \rbrace.
\end{equation}
\item Convention: when integrating a function $\phi$ defined on a space $X$ w.r.t. to a measure $\mu$ defined on a product space $X \times Y$, the integral $\int \phi \, d\mu$ is to be understood as $\int \phi \otimes 1 \, d \mu$ where $$(\phi \otimes 1 )(x,y)= \phi(x) \, \, \forall \,\, (x,y) \in X \times Y.$$

\end{itemize}
\subsubsection*{List of constants used in the paper}
\begin{itemize}
\item The constants $c_1, \dots , c_6$ are  defined in \cite{blms1}.
\item $c_7 $ \, (defined Lemma \ref{epsilonspherebarrier}).
\item $c_{8}= \omega_n R^n + \max \lbrace 2^n \omega_n, 2n\omega_n (R+1)^{n-1} \rbrace$ \, (defined in Lemma \ref{volumechange}).
\end{itemize}

\section{Generalities on varifolds theory and measure theory}

In this section we recall the necessary definitions and properties of measure theory that we will need in the sequel.
We work in the varifold setting, see \cite{menne2017concept} for a simple and concise introduction, and \cite{simon, all, buet} for more details.\\
\begin{dfn}(Varifolds)\label{varifolds}
A $d$-varifold in $\R^n$ is a nonnegative Radon measure on $\R^n \times \G$.
\end{dfn}
\begin{itemize}
\item  $V_d(\R^n)$ will denote the space of $d$-varifolds in $\R^n$.
\end{itemize}
We also consider the projection measure on $\R^n$. We call it the ``mass measure'' of the varifold $V$ and we denote it by $\|V\|$. It is defined as follows:
\begin{equation}
\|V\|(\phi):=\int_{\R^n\times \G}\phi(x) \, dV(x,S) \quad \forall \phi \in C_c(\R^n,\R^+).
\end{equation}
In other words, for every Borel set $A\subset \R^n$, we have : $\|V\|(A)=V(A\times \G)$.\\

We recall the definition of the bounded Lipschitz distance for Radon measures.
\begin{dfn}(Bounded Lipschitz distance)
Let $(X,d)$ be a locally compact separable metric space. Let $\nu$ and $\mu$ be two finite Radon measures on $X$. We define:
\begin{equation}\label{dfnboundedlip}
\Delta(\nu,\mu) := \sup\limits \Big\lbrace \Big| \int \phi(x) \, d\nu(x) - \int \phi(x) \, d\mu(x) \, \Big|\, \phi \in C^{0,1}(X,\R^+) \, \text{and} \, \max\lbrace \|\phi\|_{\infty}, \lip(\phi) \rbrace \leq1 \Big\rbrace
\end{equation}
\end{dfn}
%


We will define the notion of convergence we will be dealing with throughout this paper.
\begin{dfn}[Weak-$*$ convergence]
 Let $(X,d)$ be a locally compact and separable metric space. Let $(\mu_i)_{i\in\N}$, $\mu$ be Radon measures on $X$. We say that: 
 $(\mu_i)_i$ converges weakly-$*$ to $\mu$ if 
\begin{equation}
   \int \phi \, d\mu_i \rightarrow \int \phi \, d \mu\,
\end{equation}
for every $\phi \in C_c^0(X,\R^+)$.
\end{dfn}

The following result is a general fact in measure theory (see for instance \cite[Theorem 5.9]{st})
\begin{prop}\label{metrization1}
 Let $(X,d)$ be a locally compact separable metric space. Let $(\mu_i)_{i \in \N}$ and $\mu$ be finite Radon measures.
Assume that the measures $(\mu_i)_{i\in\N}$ and $\mu$  are supported in a compact set of $X$. 
Then:
 \begin{center}
  $\mu_i$ \text{converges weakly-$*$ to} $\mu$ \quad $\Leftrightarrow$  $\quad \Delta(\mu_i,\mu) \xrightarrow[i\rightarrow \infty]{} 0.$
 \end{center}
\end{prop}
We will define the pushforward operation on varifolds as in \cite[Section1.4]{ton}.
\begin{dfn}(Pushforward of a varifold) \label{dfn:varifoldPushforward}
Let $f$ be a $C^1$ diffeomorphism of $\R^n$. For every $\phi \in C_{c}(\R^n\times\G, \R^+)$. Define:
\begin{equation} \label{eq:varifoldPushforward}
  f_{\#}V(\phi):=\int_{\R^n\times\G} \phi(f(x),Df(x)(S))J_Sf(x)\, dV(x,S),
\end{equation}
where $Df(x)(S)$ is the image of $S$ in $\G$ by the linear isomorphism $Df(x)$ and the tangential Jacobian $J_Sf(x)$ is the determinant of the isomorphism $Df(x)$ from $ S$ to  $Df(x)(S)$, i.e.
if we write $\tilde{S}= \left( s_1 | \dots | s_d \right)^t \in \cM_{d,n}$ where $\lbrace s_i \rbrace_{i=1}^d$ is an orthonormal basis of $S$, we have
\begin{equation} \label{eq:DfS}
 Df(x)(S) = Y(Y^tY)^{-1}Y^t \quad \text{where} \,\, Y = Df(x)\circ \tilde{S}^t
\end{equation}
and 
\begin{equation} \label{eq:JSf}
J_Sf(x) := \det\left( Y^tY \right)^{\frac{1}{2}}
\end{equation} where $Df(x)\circ \tilde{S}^t$ is the product of the matrices $Df(x)\in \cM_{n, n}$ and $\tilde{S}^t\in \cM_{n, d}$.
\end{dfn}

Note that in the above definitions of $Df(x)(S)$  and  $J_S$, the choice of $\tilde{S}$ is not unique (we could take any orthonormal basis of $S$), however $Df(x)(S)$ and  $J_S$ only depend on $S$. 
%
We now recall the notion of the first variation of a varifold (and its weighted version). It expresses how the mass of the varifold changes under the pushforward operation.

\begin{dfn}[First variation]
Let $X \in C^1(\R^n, \R^n)$, $V \in V_d(\R^n)$ of finite mass, we denote 
\begin{equation}\label{firstvar}
\delta V (X) : = \int_{\R^n \times \G} \mdiv_S(X)(x)\, dV(x,S).
\end{equation}
\end{dfn}
\noindent The map 
\begin{equation}
 \delta V(\cdot) : \xC^1(\R^n,\R^n) \mapsto \R
\end{equation}
is called the \textit{first variation} of the varifold $V$, it measures the variations of the mass of the varifold. In general, for $\phi \in \xC^1(\R^n, \R^+)$ and $X \in \xC^1(\R^n, \R^n)$   we denote: 
\begin{equation}\label{wfirstvar}
     \delta (V,\phi)(X) := \int_{\R^n \times \G} \phi(x) \,\mdiv_S(X)(x) \, dV(x,S) + \int_{\R^n\times\G} \nabla\phi(x) \cdot X(x) \,dV(x,S),
\end{equation}
the map: 
\begin{equation}
 \delta (V,\cdot)(\cdot) : (\xC^1(\R^n,\R^+), \xC^1(\R^n,\R^n))\mapsto \R,
\end{equation}
is called the \textit{weighted first variation} of the varifold $V$, it expresses the change of the weighted mass of the varifold.

When the first variation "$\delta V$" of a varifold $V$ is bounded, i.e. $|\delta V(X)| \leq \| X\|_{\infty}, \, \forall X\in C^0(\R^n,\R^n)$  it can be represented by Riesz's theorem to give a vectorial Radon measure, then one would take the Radon derivative w.r.t. $\V$ and obtain the vector $h = -\frac{\delta V}{\V} \in L^1 (\R^n,\R^n, \|V\| )$ that we call the mean curvature of $V$. We have the following:
\begin{equation}
 \delta V(X) = -\int_{\R^n} X(y) \cdot h(y) \, d\|V\|(y) + (\delta V)_{s} \,\, \forall X \in C^1(\R^n, \R^n)
\end{equation}
where $(\delta V)_{s}$ is a vectorial measure, singular w.r.t. $\|V\|$.

\noindent However, $\delta V$ is not always bounded (for instance, when $\supp V$ is a discrete set of points), in this case, several definitions of approximate mean curvature vector were proposed, in our previous work \cite{blms1} and in the current work, we are interested in the definition given by Kim and Tonegawa in \cite{kt} after Brakke (\cite{brakke}).

\section{Approximate mean curvature flows of general varifolds,
and their limit spacetime Brakke flows}
In this section we recall the main definitions and results of \cite{blms1}. We start with the construction and the stability property of the "time-discrete approximate MCF" (\ref{damcf} and \ref{stability_combined}). The flow depends on an initial data in $V_d(\R^n)$, a subdivision $\cT$ and a smoothing scale $\e$. Thanks to the stability property, we can let $\delta(\cT) \rightarrow0$ (the step of the subdivision) to obtain a limit flow that we called "the approximate MCF" (Theorem \ref{damcfconvergence}), it is denoted $\left(V_{\e}(t)\right)_{t\in[0,1]}$. We then prove that  $ V_{\e}(t)\otimes dt$ converges to a limit flow that we call "spacetime Brakke flow" (Definition \ref{def:spacetimebf}).\\


We keep the same definition of the approximate mean curvature vector as in \cite{kt}. We define an approximate mean curvature vector for any varifold $V$ as follows. For any $x\in \R^n$: 
\begin{equation}\label{dfnhepsilon}
h_{\e}(x,V)=(\Phi_{\e}\ast \tilde{h}_{\e}(\cdot,V))(x)\,, \quad \text{where} \,\,\, \tilde{h}_{\e}(y,V)=-\frac{(\delta V \ast \Phi_{\e})(y)}{(\|V\| \ast \Phi_{\e})(y)+\e} \,\, \text{for any $y \in\R^n$}.
\end{equation}

We list some of the key properties of $h_{\e}$ :
\begin{prop}
Let $\e\in(0,1)$, $ M \geq 1$. Let $V\in V_d(\R^n)$ with $\|V\|(\R^n)\leq M$, there exist some constants $c_3 \geq c_2 \geq c_1 \geq 1$ depending only on $n$ such that,
\begin{equation}\label{mc_bound}
 \| h_{\e}(\cdot,V) \|_{C^1} \leq c_1 M \e^{-4}.
\end{equation}
For $\del \geq 0$, we define $f := { \rm Id} + \del\, h_{\e}$, if we assume that $\del c_2 M^3 < \e^{8}$ then $f$ is  a $C^1$ diffeomorphism and one has,
\begin{equation}\label{jacobian_bound}
 |J_{S}f(x) -1 | \leq \del c_3 M\e^{-4} \quad \, \forall \, (x,S) \in \R^n\times G_{k,n}  \,\, \text{for}\,\, k \in \{ 1, \dots , n \}.
\end{equation}
For any $\phi \in C^2(\R^n,\R^+)$ we have,
\begin{equation}\label{pushforward-expansion}
 \Big| \| f_{\#}V\| (\phi) - \|V\|(\phi) - \del \delta(V,\phi)(h_{\e}(\cdot,V)) \Big| \leq \| \phi\|_{C^2} \del.
\end{equation}
\end{prop}

\noindent Straightforward computations, as done in \cite{blms1} infer that

\begin{equation}\label{massdecrease8}
 \delta V(h_{\e}(\cdot,V)) = -\int_{\R^n} \frac{|(\delta V \ast \Phi_{\e})(y)|^2}{(\|V\| \ast \Phi_{\e})(y)+ \e} \, dy\leq 0,
\end{equation}
taking $\phi \equiv 1$ in \eqref{pushforward-expansion} yields, $\| f_{\#}V\| (\R^n) \leq \| V \|(\R^n) + \del$. The mass measure of the pushforward by the map $f$ decreases the mass up to an error $\del$, this property inherits from the MCF. The latter $\del$-decay property guarantees in particular that if we push the varifold successively with velocity equal to its approximate mean curvature, the mass will increase at most linearly, this allows the following definition of the time-discrete approximate MCF.

\begin{dfn}[Time-discrete approximate mean curvature flow]
\label{damcf} 
Let $M \geq 1$, $\e\in(0,1)$. Consider a subdivision
$T=\lbrace t_i \rbrace_{i=0}^m$ of $[0,1]$ and assume
\begin{equation}\label{smallstep}
c_3 \delta(\cT)  \leq (M+1)^{-3}\e^{8} \quad \text{with} \quad \delta(\cT) = \max_{1 \leq i \leq m} {t_i - t_{i-1}} \: .
\end{equation}
for some constant $c_3$ depending only on $n$. Let $V_0 \in V_d(\R^n)$ satisfying $\|V_0\|(\R^n)\leq M$. We can define the family $\left( V_{\e,\cT} (t_i) \right)_{i = 0 \ldots m}$ by $V_{\e,\cT}(0):= V_0$ and for $i = 1, \ldots, m$,
\begin{equation}
 V_{\e,\cT}(t_{i}) := {f_i}_\# V_{\e,\cT}(t_{i-1}) \quad \text{with} \quad 
 f_i = {\rm Id}+ (t_i - t_{i-1}) h_{\e}(\cdot,V_{\e,\cT}(t_{i-1})) \: .
\end{equation}
Then we define the family $\left( V_{\e,T}(t) \right)_{t \in [0,a]}$ by linear interpolation between the points of the subdivision, and we call it a \emph{time-discrete approximate MCF}: 
\begin{equation}
 V_{\e,\cT}(t) := \left[ {\rm Id}+(t-t_i)h_{\e}(\cdot,V_{\e,\cT}(t_i))\right]_{\#}V_{\e,\cT}(t_{i}) \quad \text{if} \quad t \in[t_{i},t_{i+1}].
\end{equation}
\end{dfn}
\begin{remk}\label{remk:damcf}
 \noindent 
 \begin{itemize}
  \item[(i)] \eqref{smallstep} is a technical condition allowing to define the pushforwards ($f$ to be $C^1$ diffeomorphism) and guaranteeing at each time that the mass is bounded by $M+1$.
  \item[(ii)] The interval $[0,1]$ was chosen just to have simple computations, the time-discrete approximate MCF (and thus its limits) can be defined on any time interval, even on $[0,\infty[$ thanks to a more precise version of \eqref{pushforward-expansion} (see \cite{blms1}).
  \item[(iii)] An alternative definition for the time-discrete approximate MCF is given by simply taking the following piecewise constant extension: for $i \in \{0,1,\ldots,m-1\}$,
\begin{equation}
 V_{\e,\cT}^{pw}(t) := V_{\e,\cT}(t_{i}) \quad \text{if} \quad t \in(t_i,t_{i+1}) \: .
\end{equation}
From \cite{blms1} we know that there exists a constant $c_4$ depending only on $n$ and $M$ such that:
\begin{equation}\label{flowscoincidence}
 \Delta(  V_{\e,\cT}(t) , V_{\e,\cT}^{pw}(t)  ) \leq  c_4 \delta(\cT) \e^{-4}, \quad \forall t \in [0,1],
\end{equation}
thus both $\left( V_{\e,\cT}(t) \right)_{t\in[0,1]}$ and 
$\left( V_{\e,\cT}^{pw}(t) \right)_{t\in[0,1]}$ lead to the same definition of $\left( V_{\e}(t) \right)_{t \in[0,1]}$. The second definition is called the \it{piecewise constant approximate flow} and it is well-suited to prove the results of Section \ref{sec:nontriviality}.\\
 \end{itemize}
\end{remk}

We recall the $\e$-Brakke-type inequality satisfied by discrete approximate MCFs that appears in \cite{blms1} and follows from estimate \eqref{pushforward-expansion}.
\begin{lemma}
Let $\e\in(0,1)$, $M \geq 1$. Let $V_0$, $W_0$ be two varifolds in $V_d(\R^n)$ with $||V_0||(\R^n) \leq M$. Let $\cT=\lbrace t_i \rbrace_{i=1}^m$ be a subdivision of $[0,1]$ satisfying \eqref{smallstep}. Let $(V^{pw}_{\e,\cT}(t))_{t\in[0,1]}$ be the piecewise constant approximate MCF w.r.t. to $\cT$  starting from $V_0$ (Remark (iii) \ref{remk:damcf}).

\noindent For any $\phi \in \xC^2(\R^n \times [0,1],\R^+)$  and $0 \leq t_1 \leq t_2 \leq 1$ we have 
\begin{multline}\label{avoidance_timeepsilonbrakke1bis}
 \left| \|V^{pw}(t_2)\|(\phi(\cdot, t_2)) - \|V^{pw}(t_1)\|(\phi(\cdot, t_1)) - \int_{t_1}^{t_2} \delta (V^{pw}(t),\phi ( \cdot, t) )(h_{\e}(t)) \: dt \right. \\ \left. - \int_{t_1}^{t_2} \int_{\R^n}  \partial_t \phi (\cdot, t)  \, d \| V^{pw}(t) \| dt \right| \leq c_5 \|\phi\|_{\xC^2} \delta(\cT)\e^{-8}
\end{multline}
where $c_5$ depends only on $n$ and $M$.
\end{lemma}

The next proposition encompasses the obtained results on the  stability of the time-discrete approximate MCF w.r.t. the subdivision and the initial data.

\begin{prop}
\label{stability_combined}
Let $\e\in(0,1)$, $M \geq 1$. Let $V_0$, $W_0$ be two varifolds in $V_d(\R^n)$ with $||V_0||(\R^n) \leq M$ and $||W_0||(\R^n)\leq M$. Let $\cT_1=\lbrace t_i \rbrace_{i=1}^m$ and $\cT_2=\lbrace s_j \rbrace_{j=1}^{m'}$ be two subdivisions of $[0,1]$ satisfying \eqref{smallstep}. Let $V_{\e,\cT_1}(t)$ (resp.$W_{\e,\cT_2}(t)$) be the discrete approximate MCF w.r.t. to $\cT_1$ (resp. $\cT_2$) starting from $V_0$ (resp. $W_0$).\\
If we set
\begin{equation}
\delta = \max \big\lbrace  \delta(\cT_1) , \delta(\cT_2) \big\rbrace ,
\end{equation}
Then, there exists a constant  $c_6$  depending on $n$ and $M$ such that
\begin{equation}
\Delta(V_{\e,\cT_1}(t),W_{\e,\cT_2}(t)) \leq  \Delta(V_0,W_0)\exp(t c_6 \e^{-n-7})+ c_{6} t \delta  \e^{-n-11} \exp(tc_6\e^{-n-7}) ,
\end{equation}
for all $t\in [0,1]$.
\end{prop}
%

%

Thanks to the stability property we managed to show the  convergence of the flow  $V_{\e,\cT_j}(t)_{t\in[0,1]}$ (for $\e$ fixed) to a \underline{unique} limit when the time step tends to $0$.

\begin{theo}[Convergence]\label{damcfconvergence}
Let $\e \in (0,1)$. Let $V_0 \in V_d(\R^n)$ of compact support. Let $\lbrace \cT_j \rbrace_{j\in\N}$ be a sequence of subdivisions of the interval $[0,1]$ with step tending to $0$ and consider $V_{\e,\cT_j}(t)_{t\in[0,1]}$, the discrete approximate MCF w.r.t. $\cT_j$ starting from $V_0$.\\
Then, as $j\rightarrow \infty$, $V_{\e,\cT_j}(t)$ converges on $[0,1]$ to a unique limit (independent of the sequence $\lbrace \cT_j \rbrace_{j\in\N}$), we call the limit the approximate MCF of $V_0$ and we denote it by $V_{\e}(t)$.\\
The convergence takes place in both the weak-* topology and in the bounded Lipschitz topology as the supports of $(V_{\e}(t))_{t\in[0,1]}$ are compact for all $t\in[0,1]$, this is due to the compactness of $\supp V_0$ and the uniform bound on the norm of $h_{\e}$ (see \eqref{mc_bound}). \\
In addition, the limit satisfies a Brakke equality w.r.t. $h_{\e}$, i.e.
\begin{equation}\label{timeepsilonbrakke}
||V_{\e}(t_2)||(\phi(\cdot,t_2)) - ||V_{\e}(t_1)||(\phi(\cdot,t_1)) = \int_{t_1}^{t_2} \delta (V_{\e}(t),\phi) (h_{\e}(\cdot,V_{\e}(t)))dt +\int_{t_1}^{t_2}\int_{\R^n} \partial_t\phi(\cdot,t) \,d||V_{\e}(t)||dt
\end{equation}
for all $\phi \in C^1(\R^n\times[0,1],\R^+)$ and $0\leq t_1 \leq t_2 \leq 1$.\\
\end{theo}

Let $V_0 \in V_d(\R^n)$ of compact support, $\e \in (0,1)$, define $\left(V_{\e}(t)\right)_{t\in[0,1]}$ to be the approximate MCF starting from $V_0$. Similarly to the previous works of Brakke, Kim and Tonegawa, we are able to prove that there exists a sequence $\lbrace \e_j \rbrace_j$, independent of $t$, for which $||V_{\e_j}(t)||$ converges $\forall t \in [0,1]$ to a Radon measure on $\R^n$, denoted $\mu(t)$. This is due to the fact that  $||V_{\e_j}(t)|| (\R^n) \leq \| V_0 \|(\R^n) \, \forall t \in [0,1] $ and an approximate continuity property of the mass measures $||V_{\e_j}(t)|| (\R^n) $ w.r.t. $t$. We are not able to prove the same result for the sequence of varifolds $V_{\e}(t)$ due to the poor understanding of the behaviour of $V_{\e}(t)$ projected on the Grassmannian component. We then considered studying the measure $\lambda_{\e} = V_{\e}(t)\otimes dt$ that we proved to be converging as $\e \rightarrow 0$ (under extraction) to a {\it spacetime Brakke flow} (see Definition \ref{def:spacetimebf}). We first define the notion of the {\it spacetime mean curvature}:

%

\begin{dfn}[Spacetime first variation and mean curvature]
\noindent \label{spacetimefrstvar}

\noindent Let $X\in C^1(\R^n \times [0,1], \R^n)$, $\lambda$ a finite Radon measure on $(\R^n \times \G \times [0,1])$. We define the spacetime first variation of $\lambda$ in the direction $X$ by: 
\begin{equation}
 \delta \lambda(X)= \int_{ \R^n \times \G \times [0,1] } \mdiv_S X(y,t) \, d\lambda (y,S,t).
\end{equation}
If the functional $\delta \lambda : C^1(\R^n \times [0,1], \R^n) \rightarrow \R$ is bounded with respect to $\|\cdot\|_{\infty}$ then by Riesz representation theorem and Radon-Nikodym decomposition, we can assert the existence of a vector field  $h(\cdot,\cdot,\lambda)$ in $ L^1 (\R^n \times [0,1], \R^n, d\mu(t)\otimes dt)$ such that  
\begin{equation}\label{st_firstvar}
 \int_{ \R^n \times \G \times [0,1] }  \mdiv_S X(y,t) \, d\lambda(y,S,t)
= -\int_{\R^n \times [0,1]} X(y,t) \cdot h(y,t,\lambda) \, d\|\lambda\| + (\delta \lambda)_s (X)
\end{equation} 
$ \forall X \in C^1(\R^n \times [0,1],\R^n)$, where $\|\lambda\|= \Pi_{\#}\lambda$, $\Pi$ being the canonical projection from $ \R^n \times \G \times [0,1]$ to $\R^n \times [0,1]$, and $(\delta \lambda)_s$ is a vector-valued Radon measure singular w.r.t. $\|\lambda\|$. The vector $h(\cdot,\cdot,\lambda)$ is then called the spacetime mean curvature of $\lambda$.

\end{dfn}

We now give the definition of  spacetime Brakke flows.

\begin{dfn}\label{def:spacetimebf}
Let $\lambda$ be a finite Radon measure on $\R^n \times \G \times [0,1]$. $\lambda$ is called a spacetime Brakke flow if: 
\begin{itemize}
 \item[$(i)$] There exist $\left( \mu(t) \right)_{t\in[0,1]}$, a family of Radon measures on $\R^n, \, \forall t \in [0,1]$, we call it the mass measure of $\lambda$, and $\nu_{(x,t)}$  a family of  probability measures for $(x,t) \in \R^n \times [0,1]$ such that $\lambda= \mu(t) \otimes \nu_{(x,t)} \otimes dt$.
 \item[$(ii)$] $\delta \lambda $ is bounded and $(\delta \lambda)_s = 0$.
 \item[$(iii)$] (Integral Brakke inequality). For any $\phi \in C_c^1(\R^n \times [0,1], \R^+)$, $0 \leq t_1 \leq t_2 \leq 1$ we have
 \begin{equation}\label{intbrakkeineq}
  \begin{split}
 \mu(t_2)(\phi(\cdot,t_2)) - &\mu(t_1)(\phi(\cdot,t_1))  \leq -\int_{t_1}^{t_2}\int_{\R^n} \phi(y,t)|h(y,t,\lambda)|^2 \, d\mu(t)(y)dt
\\& +\int_{t_1}^{t_2}\int_{\R^n \times \G } S^{\perp}(\nabla\phi(y,t))\cdot h(y,t,\lambda) \, d\lambda(y,S,t) + \int_{t_1}^{t_2}\int_{\R^n} \partial_t \phi(\cdot,t)\, d\mu(t)dt
\end{split}
 \end{equation}
\end{itemize}
where $h(\cdot,\cdot,\lambda)$ is the spacetime mean curvature of $\lambda$. We say that $\lambda$ starts from $V_0 = \mu(0) \otimes \nu_{(x,0)}$.
\end{dfn}
\begin{remk}\label{remk:spacetimebf} We make a remark on the definition of the spacetime Brakke flow and two important direct consequences:
\begin{itemize}
\item[$(i)$] One could assume only that $\lambda = V(t) \otimes dt$, with $\left( V(t) \right)_{t\in[0,1]}$ being a family of measures on $\R^n \times \G$. Actually, using Young's disintegration theorem ~\cite[Theorem 2.28]{afp} we infer that there exists a family of probability measures on $\G$, $\lbrace \nu_{(x,t)} \rbrace_{(x,t)} $, such that: 
\begin{equation}
 V(t) = \| V (t)\|\otimes \nu_{(x,t)}, 
\end{equation}
 \item[$(ii)$](Mass decay)
 \begin{equation}\label{mutdecay}
  \mu(t_2)(\R^n) \leq \mu(t_1)(\R^n) \leq \mu(0)(\R^n)
\,\,\, \text{for all} \,\,\, 0 \leq t_1 \leq t_2 \leq 1.
\end{equation}
 \item[$(iii)$]($L^2$ bound) $h \in L^2(\|\lambda\|)$ and,
 \begin{equation}
\int_{0}^1 \int_{\R^n} |h(y,t,\lambda)|^2 \, d\mu(t)(y) dt \leq \mu(0)(\R^n) = \|V_0\|(\R^n).
\end{equation}
\end{itemize}
\end{remk}
\begin{proof}[Proof of $(ii)$ and $(iii)$]
Define for every $r \in \R^{+}$, a $C^1(\R^n, \R^+)$ function  as:
\[ \phi_r = \left\{
\begin{array}{ll}
      1 \quad x \in B_r, \\
      0 \quad x \in B_{3r}^c,
\end{array} 
\right. \]
and $\| \nabla \phi_r \|_{\infty} \leq r^{-1}$. Plugging $\phi_r$ in \eqref{intbrakkeineq} we obtain (denoting $h := h(\cdot, \cdot,\lambda)$ for simplicity)  
\begin{equation}
  \begin{split}
 \mu(t_2)(\phi_r) - \mu(t_1)(\phi_r)  &\leq -\int_{t_1}^{t_2}\int_{\R^n} \phi_r|h|^2 \, d\mu(t)dt
 +\int_{t_1}^{t_2} \int_{\R^n \times \G } S^{\perp}(\nabla\phi_r)\cdot h \, d\lambda
 \\& \leq -\int_{t_1}^{t_2}\int_{B_r} |h|^2 \, d\mu(t)dt+ r^{-1}\| h\|_{L^1(d\|\lambda\|)}
\end{split}
 \end{equation}
We now let $r\rightarrow +\infty$, we obtain
\begin{equation}
 \mu(t_2)(\R^n) + \lim_{r\rightarrow +\infty} \int_{t_1}^{t_2}\int_{B_r} |h|^2 \, d\mu(t)dt \leq  \mu(t_1)(\R^n).
\end{equation}
This proves the decay property and the desired $L^2$ bound.
\end{proof}

\noindent We recall the main result on the limit of $\lambda_{\e} = V_{\e}(t)\otimes dt$.

\begin{remk}[Brakke flows and spacetime Brakke flows]
Let $(V(t))_{t\in[0,1]}$ be a Brakke flow (or particularly a MCF), $\lambda=(V(t))_{t\in[0,1]} \otimes dt$ is a spacetime Brakke flow where  $$\mu(t)=\|V(t)\| \,\, \text{and} \,\,\, h(y,t,\lambda) = h(y,V(t)) \quad \forall (y,t)\in \supp \|V(t)\| \times [0,1] .$$
\end{remk}

\begin{theo}\label{limitstbrakke}
Let $V_0 \in V_{d}(\R^n)$ of bounded support and define 
$\lambda_{\e}= V_{\e}(t)\otimes dt$. 
\noindent There exists a sequence $\e_j \xrightarrow[j\rightarrow0]{} 0$ and a spacetime Brakke flow $\lambda$ starting from $V_0$ and satisfying:
$$ \lim_j \lambda_{\e_j} = \lambda = \mu(t) \otimes \nu_{(x,t)} \otimes dt \quad \text{and} \,\, \lim_j \|V_{\e_j}(t)\| = \mu(t) \, \, \forall t \in [0,1],$$ 
where $\nu_{(x,t)}$ is a family of probability measures on $\G$ for all $(x,S) \in \R^n \times \G$.
\end{theo}
\begin{dfn}[limit spacetime Brakke flows]\label{def:limitstbrakke} We call a "limit" spacetime Brakke flow every limit of measure of the form $\left( V_{j}(t) \otimes dt \right)_{j\in\N}$ where $(V_j(t))_{t\in[0,1]}$ is a discrete (or time-discrete) approximate MCF (see Definition \eqref{damcf} and Theorem \eqref{damcfconvergence} for definitions).
\end{dfn}
\begin{remk}[Non uniqueness of the limit spacetime Brakke flows]

\noindent We recall that the limit measure $\lambda$ in Theorem \ref{limitstbrakke} depends on the choice of the subsequence $\{\e_j\}_{j\in \N}$, hence is not unique ( a priori). This, somehow, is due to the non-uniqueness of Brakke flows, as Brakke flows themselves are spacetime Brakke flows when coupled with the measure ``$dt$''.
\end{remk}

\section{Avoidance and approximate avoidance principles for spacetime Brakke flows and their approximations, consequences }\label{sec:nontriviality}

\noindent In this section we aim to prove certain avoidance and approximate avoidance principles for spacetime Brakke flows and their approximations defined in \cite{blms1}. We start with the avoidance principle w.r.t. spheres evolving by the mean curvature flow (Proposition \ref{externalvar} and Proposition \ref{internalvar}). We then prove a $\e$- avoidance principle w.r.t. spheres for the piecewise discrete approximate MCF, this allows to prove the nontriviality of the {\it limit} spacetime Brakke flow (Theorem \ref{thm:nontriviality}) when starting from the boundary of an open partition of $\R^n$ (Definition \ref{openpartition}). Theorem \ref{avoidance} concludes the section, it shows the general avoidance principle w.r.t. smooth mean curvature flow of codimension $1$ satisfied by general (not necessarily limit) spacetime Brakke flows of codimension $1$. The theorem has several consequences, notably, the inclusion of the mass measure of a spacetime Brakke flow in the level set flow of the initial varifold (the starting point varifold).

\subsection{Avoidance of the evolving spheres}
We prove the avoidance principles with respect to the evolving spheres by MCF, once when the varifold is in the exterior of the sphere (Proposition \ref{externalvar}) and another when it is in the interior (Proposition \ref{internalvar}).\\

We state the following technical lemma that will help later in proving different avoidance principles.

\begin{lemma}\label{technical}
Let $h\in (\R^n,\R^n)$, $ \phi\in C^1(\R^n,\R^{+*})$ and $S\in \G$ we have:
\begin{equation}
 -|h|^2\phi + S^{\perp}\nabla\phi\cdot h \leq \frac{1}{4}\frac{|S\nabla\phi|^2}{\phi}+\nabla\phi \cdot h
\end{equation}
\end{lemma}
\begin{proof}
We can simply write, as $\phi$ does not vanish,
\begin{equation}
\begin{split}
 -|h|^2\phi+S^{\perp}\nabla\phi\cdot h &=  -|h|^2\phi-S\nabla\phi\cdot h+\nabla\phi \cdot h 
 = -\Big|\phi^{\frac12}h+\frac12\frac{S\nabla\phi}{\phi^{\frac12}}\Big|^2+\frac{1}{4}\frac{|S\nabla\phi|^2}{\phi}+\nabla\phi \cdot h
\\& \leq \frac{1}{4}\frac{|S\nabla\phi|^2}{\phi}+\nabla\phi \cdot h \end{split}\end{equation}
this completes the proof.
\end{proof}

We introduce the notion of barrier functions and we highlight one of their main properties. For convenience reasons, we use the definition of Brakke (\cite[Section3.6]{brakke}) for barrier functions, however,  the notion of barrier functions refers to a more general class of functions (see \cite{bel}).

\begin{dfn}
A function $\psi \in C^2(\R^n \times \R^+,\R^+)$ is called a barrier function if there exists $a \in \R^n$ and $\gamma \in C^2(\R,\R^+)$ such that $$\psi(t,x)=\gamma(|x-a|^2+2dt)$$ for all $(x,t) \in \R^n \times \R^+$ and 
$$ \left(\frac{d}{dr} \gamma(r)\right)^2 \leq 4 \gamma(r) \frac{d^2}{dr^2}\gamma(r)
$$ for all $r\in\R$.
\end{dfn}
\begin{lemma}\label{spherebarrier}(Sphere barriers)
Let $\psi$ be a barrier function, for every $S\in\G$ we have 
\begin{equation}
 \frac{1}{4}\frac{|S\nabla\psi|^2}{\psi}+S:\nabla^2\psi + \partial_t\psi \leq 0 
\end{equation}
on $\lbrace (x,t) \in \R^n\times \R , \psi(x,t) \neq 0 \rbrace$. 
\end{lemma}
\begin{proof}
 Let $\gamma \in C^2(\R,\R^+)$ such that $$\psi(t,x)=\gamma(|x|^2+2dt)$$ for all $(x,t) \in \R^n \times \R^+$ where we set $a=0$ for simplicity. We compute the derivatives of $\psi$ as:
 \begin{doublespacing}
 \begin{equation}
\begin{array}{ll}
\psi(x,t)&=\gamma(|x|^2+2dt),\\
\partial_t\psi(x,t)&=2d\gamma'(|x|^2+2dt),\\
\nabla\psi(x,t) &= 2\gamma'(|x|^2+2dt)x, \text{and}\\
\nabla^2\psi(x,t) &= 4\gamma''(|x|^2+2dt) x \otimes x + 2 \gamma'(|x|^2+2dt)I_n
\end{array}
\end{equation}
\end{doublespacing}
We have then, on $\lbrace (x,t) \in \R^n\times \R , \psi(x,t) \neq 0 \rbrace$, using $S:I_n=\tr(S)=d$ and $ \gamma'(r)^2 \leq 4 \gamma(r)\gamma''(r)$ for all $r\in\R$,
\begin{equation}
\begin{split}
&\frac{1}{4}\frac{|S\nabla\psi(x,t)|^2}{\psi(x,t)} -S:\nabla^2\psi(x,t) + \partial_t \psi(x,t) \\&= \frac{|S(x)|^2\gamma'(|x|^2+2dt)^2}{\gamma(|x|^2+2dt)} -  4 |S(x)|^2\gamma''(|x|^2+2dt)-2d \gamma'(|x|^2+2dt)+2d\gamma'(|x|^2+2dt)
\\& = |S(x)|^2 \left( \frac{\gamma'(|x|^2+2dt)^2}{\gamma(|x|^2+2dt)} -  4 \gamma''(|x|^2+2dt)\right) \leq 0.
\end{split}
\end{equation}
and this completes the proof.
\end{proof}

We now possess the necessary tools to prove the avoidance principle with respect to the evolving spheres. The proof consists of injecting a suitable barrier function in the spacetime Brakke inequality and using  Lemmas \ref{technical} and \ref{spherebarrier}.

\begin{prop}(Sphere barrier to external varifolds) \label{externalvar}
Let $(\mu(t))_{t\in[0,1]}$ be the mass measure of a spacetime Brakke flow (Definition \ref{def:spacetimebf}) starting from a varifold $V_0 \in V_d(\R^n)$. We have:
\begin{equation}
 \text{if} \,\, \mu(0)(B(a,R))=0 \quad \text{then} \,\, \mu(t)(B(a,(R^2-2dt)^{\frac12})=0 \,\, \forall t \in (0,R^2/2d).
\end{equation}
\end{prop}

\begin{proof}
Define $\psi(x,t)=\gamma(|x-a|^2+2dt)$ such that
\begin{equation}
\gamma(r)=\left\{
 \begin{array}{ll}
(R^2-r)^{\beta} \quad  &\text{for} \,\, r\leq R^2, \\
0 \quad          &\text{for} \,\, r > R^2 
 \end{array} \right.
\end{equation}
with $\beta > 2$ (so that both $\psi$ and $\gamma$ are $C^2$), then, easy computations show that $\psi$ is a barrier function. Plugging $\psi$ into the integral Brakke inequality \eqref{intbrakkeineq} we obtain for any $t_1 \in [0,1]$ (removing the dependence on variables and noting $h=h(\cdot,\cdot,\lambda)$ for simplicity) 
\begin{equation}
 \mu(t_1)(\psi(\cdot,t_1)) - \mu(0)(\psi(\cdot,0)) \leq \int_0^{t_1}\int_{\R^n\times \G} -\psi|h|^2 + S^{\perp}\nabla \psi \cdot h + \partial_t\psi \,d\lambda
\end{equation}
We note that by assumption, $\mu(0)(\psi(\cdot,0))=0$ as $\psi(\cdot,0)$ vanishes outside $B(a,R)$. By Lemma \ref{technical} and the definition of the spacetime mean curvature (Definition \eqref{st_firstvar}) we have 
\begin{equation}
 \begin{split}
 \mu(t_1)(\psi(\cdot,t_1))  &\leq \int_0^{t_1}\int_{ \{(x,S,t), \psi\neq0\} } \frac14 \frac{|S\nabla \psi |^2 }{\psi}+ \nabla \psi \cdot h  + \partial_t\psi \,d\lambda
 \\& = \int_0^{t_1}\int_{ \{(x,S,t), \psi\neq0\} }  \frac14 \frac{|S\nabla \psi |^2 }{\psi} + S:\nabla^2 \psi  + \partial_t\psi \,d\lambda
\end{split}\end{equation}
Hence, by Lemma \ref{spherebarrier} we deduce that $\mu(t_1)(\psi(\cdot,t_1)) =0$ (as $\psi \geq 0$). By construction, $\psi(\cdot,t_1) > 0 $ on $B(a,(R^2-2dt_1)^{\frac12})$, this implies that $\mu(t_1)( B(a,(R^2-2dt_1)^{\frac12}) )=0$ and finishes the proof of the proposition.
\end{proof}

\begin{cor}[Convex set barriers]\label{convexbarrier}
Let $(\mu(t))_{t\in[0,1]}$ be the mass measure of a spacetime Brakke flow (Definition \ref{def:spacetimebf}) starting from a varifold $V_0 \in V_d(\R^n)$. If $\supp \mu(0)$ is contained in a convex set $K$, then $\mu(t) \subset K$ for all $t\in[0,1]$. Equivalently, $\bigcup\limits_{t} \supp \mu(t)$ is contained in the convex hull of $\supp \mu(0)$.
\end{cor}
\begin{proof}
Using the sphere barrier to external varifolds (Proposition \ref{externalvar}), the proof is then a direct adaptation of the proof of \cite[Theorem 3.8]{brakke}.
\end{proof}

\begin{prop}[Sphere barrier to internal varifolds]\label{internalvar}
Let $(\mu(t))_{t\in[0,1]}$ be the mass measure of a spacetime Brakke flow (Definition \ref{def:spacetimebf}) starting from a varifold $V_0 \in V_d(\R^n)$.
\begin{equation}
 \text{if} \,\, \supp \mu(0) \subset \overline{B}(a,R) \quad \text{then} \,\, \supp \mu(t) \subset \overline{B}(a,(R^2-2dt)^{\frac12}) \,\, \forall t \in (0,R^2/2d).
\end{equation}
\end{prop}
\begin{proof}
The proof is similar to the proof of Proposition \ref{externalvar}.
We define a test function  $\psi(x,t)=\gamma(|x-a|^2+2dt)$ such that
\begin{equation}
\gamma(r)=\left\{
 \begin{array}{ll}
 0 \quad  &\text{for} \,\, r\leq R^2, \\
(r-R^2)^{\beta} \quad          &\text{for} \,\, r > R^2 
 \end{array} \right.
\end{equation}
with $ \beta > 2$ (so that both $\psi$ and $\gamma$ are $C^2$). Construct  a $C^2$ function $\tilde{\psi}$ equal to $\psi$ on $B(a,2R)$ and to $0$ outside $B(a,3R)$. By the convex barrier principle, $\bigcup_{t\in[0,1]} \supp \mu(t) \subset \overline{B}(a,R)$ thus, on  $\supp \mu(t) \cap \{ \tilde{\psi} > 0\}$ we have the formula
\begin{equation}\label{spherebarrierbis}
 \frac{1}{4}\frac{|S\nabla\tilde{\psi}|^2}{\tilde{\psi}}+S:\nabla^2\tilde{\psi} + \partial_t\tilde{\psi} \leq 0 
\end{equation}
Finally, by the spacetime Brakke inequality \eqref{intbrakkeineq}, Lemma \ref{technical} and \eqref{spherebarrierbis}, we infer that for any $t\in[0,1]$
\begin{equation}
 \mu(t)(\tilde{\psi}(\cdot,t)) \leq \mu(0)(\tilde{\psi}(\cdot,0)) =0 \quad \text{and } \,\, \mu(t)(\tilde{\psi}(\cdot,t)) = 0.
\end{equation}
Hence, as $\tilde{\psi}(\cdot,t) > 0 $ on $B(a,2R) \setminus \overline{B}(a,(R^2-2dt)^{\frac12})$ and $\supp \mu(t) \subset \overline{B}(a,R)$ we infer that $\supp \mu(t) \subset \overline{B}(a,(R^2-2dt)^{\frac12})$ and we finish the proof of the proposition.
\end{proof}

\begin{cor}[Non-existence of Compact stationary varifolds]
We call a varifold $V\in V_d(\R^n)$ stationary if: $$ \forall \, X\in C_c^1(\R^n,\R^n), \quad\delta V(X) = 0. $$
It says in particular that $h_{\e}(\cdot,V) =0$ for every $\e \in (0,1)$ and that the space-time Brakke flow of $V$, which is constructed as a limit of discrete approximate MCFs, is $V(t)\otimes dt$ with $V(t)=V, \, \forall t \in [0,1]$.

\noindent The sphere barrier to internal varifolds property (Proposition \ref{internalvar}) implies that there is no compact stationary varifold of any codimension. In particular, it proves that there is no compact minimal surface (a stationary Brakke flow in general), of any codimension, in $\R^n$ (which was already known).
\end{cor}

\subsection{Nontriviality of codimension $1$ { \it limit} spacetime Brakke flows} \label{subsec:nontriviality}
We aim to prove that limit spacetime Brakke flows (Definition \ref{def:limitstbrakke}) are nontrivial when the starting varifold is the boundary of an open partition (Definition \ref{openpartition}). The sense we give to the nontriviality is that $\mu(t)(\R^n) > 0$ for a certain time interval $[0,t_0], t_0 >0$. The proof relies on Lemma \ref{epsilonspherebarrier} on the approximate avoidance principle w.r.t. the evolving spheres satisfied by the time-discrete approximate MCF.\\

\noindent In order to prove the nontriviality of spacetime limit spacetime Brakke flows: we first extract a sufficient relation between $\delta(\cT)$ and $\e$ allowing to prove the convergence of $\left(\|V_{\e,\cT}(t)\|\right)_{t\in[0,1]}$ to $(\mu(t))_{t\in[0,1]}$. Let $(V_{\e,\cT_1}(t))_{t\in[0,1]}$ and $(V_{\e,\cT_2}(t))_{t\in[0,1]}$ be two time-discrete approximate MCFs associated with two subdivisions $\cT_1, \cT_2$ satisfying \eqref{smallstep} and starting from a varifold $V_0 \in V_d(\R^n)$. We know from \eqref{stability_combined} that for all $t\in[0,1]$ and for $\delta = \max \{ \delta(\cT_1), \delta(\cT_2) \}$,
\begin{equation}
\Delta(\|V_{\e,\cT_1}(t)\|,\|V_{\e,\cT_2}(t)\|) \leq  c_{6} \delta  \e^{-n-11} \exp(c_6\e^{-n-7}).
\end{equation}
We let $\delta(\cT_1) \rightarrow 0$ and infer from Theorem \ref{damcfconvergence} that for all $t\in[0,1]$, 
\begin{equation}
\Delta(\|V_{\e,\cT_1}(t)\|,\|V_{\e}(t)\|) \leq  c_{6} \delta(\cT_1)  \e^{-n-11} \exp(c_6\e^{-n-7}).
\end{equation}
Let $\{ \e_j \}_{j\in\N} \rightarrow0$ such that $\|V_{\e_j}\|(t) \xrightharpoonup[]{*} \mu(t) , \, \forall t\in[0,1]$, where $(\mu(t))_{t\in[0,1]}$ is a mass measure of a certain spacetime Brakke flow starting from $V_0$. We know from the convex barrier principle that $\supp \mu(t)$ is bounded for all $t\in [0,1]$, this implies that $\Delta(\| V_{\e_j}(t) \| , \mu(t) ) \xrightarrow[j\rightarrow \infty]{}0, \, \forall t \in [0,1]$. Now, for a sequence of subdivisions $(\cT_j)_{j\in\N}$ satisfying \eqref{smallstep} and the condition  
\begin{equation}\label{cvcondition}
 \delta(\cT_j)  \e_j^{-n-11} \exp(c_6\e_j^{-n-7}) \xrightarrow[j \rightarrow \infty ]{} 0.
\end{equation}
We can assert that $\forall t \in [0,1]$
\begin{equation}
 \Delta(\mu(t), \|V_{\e_j,\cT_j}(t)\| ) \leq \Delta(\mu(t), \|V_{\e_j}(t)\| ) + \Delta(\|V_{\e_j}(t)\|, \|V_{\e_j,\cT_j}(t)\| ) \xrightarrow[j \rightarrow \infty ]{}0
\end{equation}
Again, by the boundedness of $\supp \mu(t)$, we deduce that $\|V_{\e_j,\cT_j}(t)\|   \xrightharpoonup[j\rightarrow \infty ]{*} \mu(t), \, \forall t \in [0,1]$. Finally, from \eqref{flowscoincidence} we deduce that $\|V^{pw}_{\e_j,\cT_j}(t)\|   \xrightharpoonup[j \rightarrow \infty ]{*} \mu(t), \, \forall t \in [0,1]$.

The next lemma is the key to prove the nontriviality property. We prove a $\e$-avoidance principle w.r.t. the evolving spheres by the MCF similar to Proposition \ref{externalvar}. The proof consists of injecting a suitable test function encoding the evolution of spheres by MCF in the $\e$-Brakke-type inequality satisfied by approximate MCFs (inequality \eqref{avoidance_timeepsilonbrakke1bis}).

\noindent Note: the following lemma is valid for any codimension, we state it in the codimension $1$ case for conveniency reasons.
\begin{lemma}\label{epsilonspherebarrier}
Let $M\geq 1$, $\e \in (0,1)$ and $V_0 \in V_{n-1}(\R^n)$ with $\|V_0\|(\R^n) \leq M$. Let $\cT=\lbrace t_i \rbrace_{i=1}^{m}$ a subdivision of $[0,1]$ satisfying \ref{smallstep} and $(V_{\e,\cT}^{pw}(t))_{t\in[0,1]}$ be the piecewise constant approximate MCF w.r.t. $\cT$ starting from $V_0$.\\
Define, $\psi(x,t)=\gamma(|x-a|^2+2dt)$ such that
\begin{equation}
\gamma(r)=\left\{
 \begin{array}{ll}
(R^2-r)^4 \quad  &\text{for} \,\, r\leq R^2, \\
0 \quad          &\text{for} \,\, r > R^2 
 \end{array} \right.
\end{equation}
Assume that $c_5 \delta(\cT) \e^{-8}\leq \e$. Then, there exists $\e_0 \in (0,1)$ depending only on $n$ and $M$ such that: for any $\e \in (0,\e_0)$ we have: 
\begin{equation}
 \|V_{\e,\cT}^{pw}(t_2)\|(\psi(\cdot,t_2))  - \|V_{\e,\cT}^{pw}(t_1)\|(\psi(\cdot,t_1))  
 \leq c_7 \e^{\frac16},
\end{equation}
where $c_7$ depends only on $n$, $M$ and $R$.
\end{lemma}
\begin{proof}
We assume that $\cT$ satisfies \eqref{smallstep} to give sense to the approximate MCF. We note that choice of the power in the definition of $\gamma$ is relevant, as long as it is strictly bigger than $3$ (so that $\psi$ is $C^3$). In the proof, we denote for simplicity $V(t) = V^{pw}_{\e,\cT}(t)$.

\noindent {\bf Step $1$:} We first prove that 
\begin{equation}\label{hel2bound}
\int_{0}^1\int_{\R^n} \frac{|(\delta V(t) \ast \Phi_{\e})(y)|^2}{(\|V(t)\| \ast \Phi_{\e})(y)+ \e} \, dydt \leq 2M.
\end{equation}
Indeed, from \eqref{pushforward-expansion} for $\phi \equiv 1$, $V= V(t_i)$ and $\del = t_{i+1}-t_i$ and from \eqref{massdecrease8} we infer that $\forall i \in \{ 0, \dots , n-1 \}$ 
\begin{equation}
 \|V(t_{i+1})\|(\R^n) - \|V(t_{i})\|(\R^n) + \int_{t_i}^{t_{i+1}}\int_{\R^n} \frac{|(\delta V(t) \ast \Phi_{\e})(y)|^2}{(\|V(t)\| \ast \Phi_{\e})(y)+ \e} \, dydt \leq t_{i+1} - t_{i}
\end{equation}
Summing up the inequalities for $i \in \{ 0, \dots , n-1 \}$  we obtain 
\begin{equation}
 \|V(1)\|(\R^n) - \|V(0)\|(\R^n) + \int_{0}^{1}\int_{\R^n} \frac{|(\delta V(t) \ast \Phi_{\e})(y)|^2}{(\|V(t)\| \ast \Phi_{\e})(y)+ \e} \, dydt \leq 1 \leq M
\end{equation}
Then, step $1$ follows from $\| V(0)\|(\R^n) = \| V_0 \|(\R^n) \leq  M$.

\noindent {\bf Step $2$:} Let $\psi_{\e} = c^{-1} \left( \psi + 4 \e^{\frac16} \| \psi \|_{C^3}  \right) $ with $c=c(n,R)=2\max \lbrace \|\psi\|_{C^3} , \| \nabla \psi \|_{L^2},1 \rbrace < \infty$. We prove that $\psi_{\e} \in  \mathcal{A}_j$ and $\nabla\psi_{\e} \in \mathcal{B}_j$ with $j= \lfloor \frac12 \e^{-\frac16} \rfloor$ and for $\e \in (0,4^{-6})$.

\noindent We know that $\psi \geq 0$, and $\displaystyle 4 \e^{\frac16} j \geq 1$, this implies: 
\begin{equation}
\begin{array}{ll}
\psi(\cdot,t) \in C^3(\R^n,\R^+) \,\, \forall t \in [0,1], \\
 \|\psi_{\e}(\cdot,t)\|_{\xC^2}\leq 1 \,\, \forall t \in [0,1], \\
 \| \nabla^m \psi_{\e}(\cdot,t)\|_{\infty} =c^{-1} \|\nabla^m \psi(\cdot,t) \|_{\infty} \leq j \psi_{\e}(x,t)\leq j, \,\, \forall (x,t) \in \R^n\times [0,1] \,\, \text{and} \,\, m \in \{1,2,3\}, \\
 \| \nabla \psi_{\e}(\cdot,t) \|_{L^2(\R^n)} = c^{-1} \| \nabla \psi(\cdot,t) \|_{L^2(\R^n)} \leq 1 \leq j \,\, \forall t\in[0,1].
 \end{array}
\end{equation}
Thus $\psi_{\e}(\cdot,t) \in \mathcal{A}_j$ and $\nabla\psi_{\e}(\cdot,t) \in \mathcal{B}_j$ $\forall t \in [0,1]$.

\noindent {\bf Step $3$:} We inject $\psi_{\e}$ into the inequality \eqref{avoidance_timeepsilonbrakke1bis} and deduce the desired result.

\noindent Indeed, we have:
\begin{equation}\label{epsilonspherebarrier_prf1}
\begin{split}
  \|V(t_2)\|(\psi_{\e}(\cdot, t_2)) & - \|V(t_1)\|(\psi_{\e}(\cdot, t_1))  - \int_{t_1}^{t_2} \delta (V(t),\psi_{\e} ( \cdot, t) )(h_{\e}(t)) \: dt  
  \\& - \int_{t_1}^{t_2} \int_{\R^n}  \partial_t \psi_{\e} (\cdot, t)  \, d \| V(t) \| dt  \leq  c_5 \|\psi_{\e}\|_{\xC^2} \delta(\cT)\e^{-8} \leq \e
\end{split}
\end{equation}
We recall that  $ \mdiv_S (h_{\e})\psi_{\e} = \mdiv_S(\psi_{\e}h_{\e}) - h_{\e} \cdot \nabla\psi_{\e} $ so that 
\begin{equation}\label{epsilonspherebarrier_prf2}
\begin{split}
 &\delta (V(t),\psi_{\e}(\cdot,t))(h_{\e}(t)) = \int_{\R^n\times \G} \mdiv_S h_{\e}(t) \, \psi_{\e}(\cdot,t) + \nabla \psi_{\e}(\cdot,t) \cdot h_{\e}(t) \, dV(t)
 \\& = \delta V(t)(\psi_{\e}h_{\e}) + \int_{\R^n\times \G} S^{\perp}\nabla \psi_{\e}\cdot h_{\e}(t) \, dV(t). 
\end{split}\end{equation}
Applying \cite[proposition5.4]{kt} we have for $V=V(t)$ , $\phi = \psi_{\e}(\cdot,t)$ after we check that $j \leq \frac12 \e^{-\frac16}$ we deduce that there exists $\e_0 \in (0,4^{-6})$ depending only on $M$ and $n$ such that, for any $\e \in (0,\e_0)$ one has, for all $t\in [0,1]$
\begin{equation}
 \begin{split}
  \delta V(t)(\psi_{\e}(\cdot,t)h_{\e}(t)) \leq - \int_{\R^n} \frac{\psi_{\e}(\cdot,t)|\Phi_{\e}\ast\delta V(t)|^2 }{\Phi_{\e}\ast \|V(t)\|+\e}dx + \e^{\frac14}\left( \int_{\R^n} \frac{\psi_{\e}(\cdot,t)|\Phi_{\e}\ast\delta V(t)|^2 }{\Phi_{\e}\ast \|V(t)\|+\e}dx+1\right)
 \end{split}
\end{equation}
and that
\begin{equation}
 - \int_{\R^n} \frac{\psi_{\e}(\cdot,t)|\Phi_{\e}\ast\delta V(t)|^2 }{\Phi_{\e_j}\ast \|V(t)\|+\e}dx \leq - \int_{\R^n}  \psi_{\e}(\cdot,t) |h_{\e}(t)|^2 d\|V(t)\| + \e^{\frac14}\left(\int_{\R^n} \frac{\psi_{\e}(\cdot,t)|\Phi_{\e}\ast\delta V(t)|^2 }{\Phi_{\e}\ast \|V(t)\|+\e}dx +1 \right)
\end{equation}
thus using $\| \psi_{\e} \|_{\infty} \leq 1$ and \eqref{hel2bound} we deduce that 
\begin{equation}\label{epsilonspherebarrier_prf3}
\int_{t_1}^{t_2} \delta V(t)(\psi_{\e}(\cdot,t) h_{\e}(t)) \,dt  \leq -  \int_{t_1}^{t_2}\int_{\R^n}  \psi_{\e}(\cdot,t) |h_{\e}(t)|^2  \, d\|V(t)\| dt + \e^{\frac14} \left( 4M+2 \right)
\end{equation}
From \eqref{epsilonspherebarrier_prf1} , \eqref{epsilonspherebarrier_prf2}  and \eqref{epsilonspherebarrier_prf3} we infer that 

\begin{equation}
\begin{split}
  \|V(t_2)\|(\psi_{\e}(\cdot, t_2)) & - \|V(t_1)\|(\psi_{\e}(\cdot, t_1))  
   \leq  \int_{t_1}^{t_2}\int_{\R^n\times \G}  -\psi_{\e}(\cdot,t) |h_{\e}(t)|^2  + S^{\perp}\nabla \psi_{\e}\cdot h_{\e}(t)\, dV(t) dt  
   \\& +\int_{t_1}^{t_2}\int_{\R^n}  \partial_t \psi_{\e} (\cdot, t) \, d\|V(t)\| dt
   + \e^{\frac14} \left( 4M+3 \right)
\end{split}
\end{equation}
From Lemma \ref{technical} we infer that 
\begin{equation}\label{epsilonspherebarrier_prf4}
\begin{split}
  \|V(t_2)\|(\psi_{\e}(\cdot, t_2))  - \|V(t_1)\|(\psi_{\e}(\cdot, t_1))  
   & \leq  \int_{t_1}^{t_2}\int_{\R^n}  \frac{1}{4}\frac{|S\nabla\psi_{\e}(\cdot,t)|^2}{\psi_{\e}(\cdot,t)}+\nabla\psi_{\e}(\cdot,t) \cdot h_{\e}(t) + \partial_t \psi_{\e} (\cdot, t) \, dV(t) dt
\\& + \e^{\frac14} \left( 4M+3 \right)
   \end{split}
\end{equation}
for all $ 0\leq t_1 \leq t_2 \leq 1$, dividing by $\psi_{\e}$ is possible as  $ \psi_{\e}>0$. From \cite[Proposition5.5]{kt} for $V=V(t)$ and  $g=\nabla \psi_{\e}(\cdot,t)$ we have, reducing $\e_0$ if necessary, for any $\e \in (0,\e_0)$,
\begin{equation}
 \begin{split}
  &\Big| \int_{t_1}^{t_2}\int_{\R^n} h_{\e}(t) \cdot \nabla \psi_{\e}(\cdot,t) \, d\V(t) + \delta V(t)\left(\nabla \psi_{\e}(\cdot,t) \right) \,dt \Big| 
  \\&\leq \e^{\frac14}(t_2-t_1)+ \e^{\frac14}\int_{t_1}^{t_2} \left(\int_{\R^n} \frac{|\Phi_{\e}\ast\delta V(t)|^2}{\Phi_{\e}\ast\V(t)+\e}dx \right)^{\frac12}dt
 \\& \leq \e^{\frac14}+ \e^{\frac14}\left(\int_{0}^{1}1 dt\right)^{\frac12} \left( \int_{0}^{1}\int_{\R^n} \frac{|\Phi_{\e}\ast\delta V(t)|^2}{\Phi_{\e}\ast\V(t)+\e}dx dt\right)^{\frac12}
 \\& \leq  \e^{\frac14}(1+(2M)^{\frac12}) \leq  \e^{\frac14}(2+2M).
 \end{split}
\end{equation}
This gives,
\begin{equation}
\begin{split}
\int_{t_1}^{t_2}\int_{\R^n} h_{\e}(t) \cdot \nabla \psi_{\e}(\cdot,t) \, d\V(t)dt &\leq -\int_{t_1}^{t_2}\delta V(t)\left(\nabla \psi_{\e}(\cdot,t) \right) \,dt +  \e^{\frac14}(2+2M)
\\& =  \int_{t_1}^{t_2}\int_{\R^n} -S:\nabla^2\psi_{\e}(\cdot,t) dV(t)dt + \e^{\frac14}(2+2M)
\end{split}
\end{equation}
From \eqref{epsilonspherebarrier_prf4} we infer that 
\begin{equation}
\begin{split}
  \|V(t_2)\|(\psi_{\e}(\cdot, t_2))  - \|V(t_1)\|(\psi_{\e}(\cdot, t_1))  
   & \leq  \int_{t_1}^{t_2}\int_{\R^n}  \frac{1}{4}\frac{|S\nabla\psi_{\e}(\cdot,t)|^2}{\psi_{\e}(\cdot,t)} -S:\nabla^2\psi_{\e}(\cdot,t)  + \partial_t \psi_{\e} (\cdot, t) \, dV(t) dt
\\& + \e^{\frac14} \left( 6M+5 \right)
   \end{split}
\end{equation}
We note that $\psi_{\e} \geq c^{-1} \psi$, $\nabla^m \phi_{\e} = c^{-1} \nabla^m\phi, \, \forall m \in \lbrace 1,2 \rbrace$ and that $\psi$ is a barrier function, hence  
\begin{equation*}
 \frac{1}{4}\frac{|S\nabla\psi_{\e}(\cdot,t)|^2}{\psi_{\e}(\cdot,t)} -S:\nabla^2\psi_{\e}(\cdot,t)  + \partial_t \psi_{\e} (\cdot, t)
 \leq  c^{-1} \left(  \frac{1}{4}\frac{|S\nabla\psi(\cdot,t)|^2}{\psi(\cdot,t)} -S:\nabla^2\psi(\cdot,t)  + \partial_t \psi(\cdot, t) \right) \leq 0
\end{equation*}
this implies 
\begin{equation*}
 \|V(t_2)\|(\psi_{\e}(\cdot, t_2))  - \|V(t_1)\|(\psi_{\e}(\cdot, t_1)) \leq \e^{\frac14} \left( 6M+5 \right)
\end{equation*}
We conclude from the uniform boundedness of the mass (Remark \ref{remk:damcf}) and from $\| \psi \|_{C^2} \leq c$ that 
\begin{equation*}
 c^{-1}\|V(t_2)\|(\psi(\cdot, t_2))  - c^{-1}\|V(t_1)\|(\psi(\cdot, t_1)) \leq  4\e^{\frac16}(M+1)+  \e^{\frac14} \left( 6M+5 \right) \leq \e^{\frac16} (10M+9) 
\end{equation*}
if we set $c_7=c(10M+9) $ we deduce that 
\begin{equation*}
 \|V(t_2)\|(\psi(\cdot, t_2)) - \|V(t_1)\|(\psi(\cdot, t_1)) \leq c_7 \e^{\frac16},
\end{equation*}
and this finishes the proof.
\end{proof}

We introduce the notion of open partitions, their boundaries are the objects for which the limit spacetime Brakke flow is not trivial. The following definition is the same as \cite[Definition 4.1]{kt} with the extra condition on the compactness of the boundary.

\begin{dfn}(Open partitions)\label{openpartition}
A collection of sets $\mathcal{E}=\bigcup\limits_{i=1}^N E_i \subset \R^n$ is called an open partition if 
\begin{enumerate}
\item $E_1, E_2, \dots, E_N$ are open and mutually disjoint.
\item $\cH^{n-1}\left(\R^n \setminus \bigcup\limits_{i=1}^N E_i \right) < \infty$.
\item $\bigcup\limits_{i=1}^N \partial E_i$ is \underline{compact} and countably $(n-1)$-rectifiable.
\end{enumerate}

\end{dfn}
\begin{remk}
The fact that $\cH^{n-1}\left(\R^n \setminus \bigcup\limits_{i=1}^N E_i \right) < \infty$ implies that $ \R^n \setminus \bigcup\limits_{i=1}^N E_i  = \bigcup\limits_{i=1}^N \partial E_i $. Given an open partition $\mathcal{E}$, one can define a piecewise approximate MCF for the natural integral varifold associated to $\partial \mathcal{E} = \bigcup\limits_{i=1}^N \partial E_i $, we denote it by $(\partial\mathcal{E})_{\e,\cT}(t)=\bigcup\limits_{i=1}^N (\partial E_i)_{\e,\cT}(t)$. The open partition character is preserved through the flow as the pushforward maps involved in the construction of the flow are $C^1$ diffeomorphisms.

\end{remk}


%
In the next lemma we investigate the change of volume, restricted to fixed balls of $\R^n$, of the evolution by approximate mean curvature of an open partition. For simplicity, we state the lemma for uniform subdivisions.
\begin{lemma}[Change of volume]\label{volumechange}
Let $\e \in (0,1)$, $M\geq 1$ and $\cT$ be the uniform subdivision of $[0,1]$, of time step $\del > 0$, satisfying \eqref{smallstep}. Let $\mathcal{E}=\bigcup\limits_{i=1}^N E_i$ be an open partition of $\R^n$ with $\cH^{n-1}\left( \bigcup\limits_{i=1}^N \partial E_i \right) \leq M$.

\noindent Let $(\partial\mathcal{E})_{\e,\cT}(t)=\bigcup\limits_{i=1}^N (\partial E_i)_{\e,\cT}(t)$ be the piecewise-constant approximate MCF with respect to $\cT$ starting from $\mathcal{E}$, and  $(\mathcal{E})_{\e,\cT}(t)=\bigcup\limits_{i=1}^N ( E_i)_{\e,\cT}(t)$ 
the corresponding open partition.

%

\noindent Then, for any $i \in \{ 1, \dots, N\}$ and  $(a,R) \in (\R^n,\R^+)$ one has

\begin{equation}
 |\cL^n(B(a,R) \cap ( E_i)_{\e,\cT}(t + \del))-\cL^n\left( B(a,R) \cap ( E_i)_{\e,\cT}(t)\right)|\leq c_{8} \e.
\end{equation}
for any $t\in [0, 1-\del]$, where $c_{8}$ is a constant depending only on $n$ and $R$.
\end{lemma}
\begin{proof}
The lemma is a consequence of the general claim: 
Let $E$ be an open set in $\R^n$. Let $f$  be a $C^1$ diffeomorphism  of $\R^n$, assume that  $\delta := \max \lbrace \| f - {\rm Id} \|_{\infty}, \| Jf -1\|_{\infty} \rbrace < 1$. Let $B=B(a,R)$ for some $a\in\R^n$ and $R \in \R^+$.\\
One has:
\begin{equation}
 |\cL^n(B \cap f(E))-\cL^n(B \cap E)|\leq  \delta c_{8},
\end{equation}
where $c_{8}$ is a constant depending only on $n$ and $R$.

\noindent Indeed, we have by the area formula:
\begin{equation}
 \cL^n(B \cap f(E)) =  \int_{f(E)} \chi_B \,d\cL^n = \int_{E} \chi_{f^{-1}(B)} \, Jf \,d\cL^n
\end{equation}
This gives
\begin{equation}
\begin{split}
 |\cL^n(B \cap f(E)) -\cL^n(B \cap E)|
 &= \Big| \int_E \chi_{f^{-1}(B)} \,Jf \,d\cL^n - \int_{E} \chi_{B} \, d \cL^n \Big| 
 \\& \leq \int_{E} \chi_B \|Jf-1\|_{\infty} d \cL^n 
 + \| Jf \|_{\infty} \int_{E} |\chi_{f^{-1}(B)}-\chi_{B} | \, d \cL^n
 \\& \leq \delta \cL^n(B)+ 2 \cL^n( B(a,R+\delta) \setminus B(a,R-\delta))
\end{split}
\end{equation}
If $\delta >R$ then $\cL^n( B(a,R+\delta) \setminus B(a,R-\delta)) \leq \cL^n(B(a,2\delta))\leq \delta 2^n \omega_n$. Otherwise, by the mean value theorem applied to the function $x \mapsto x^n$ we have
\begin{equation*}
 (R+\delta)^n - (R-\delta)^n \leq 2n\delta (R+\delta)^{n-1}, 
\end{equation*}
hence,
\begin{equation}
\begin{split}
 &\cL^n( B(a,R+\delta) \setminus B(a,R-\delta)) 
\\& \leq \omega_n \left( (R+\delta)^n - (R-\delta)^n \right)
\\& \leq 2n\omega_n\delta  (R+\delta)^{n-1}\leq 2n\omega_n  \delta  (R+1)^{n-1}
\end{split}
\end{equation}
this implies that 
\begin{equation*}
|\cL^n(B \cap f(E)) -\cL^n(B \cap E)| \leq c_{8} \delta
\end{equation*}
with $c_{8}= \omega_n R^n + \max \lbrace 2^n \omega_n, 2n\omega_n (R+1)^{n-1} \rbrace$, this completes the proof of the claim.

\noindent Applying the previous claim in our context,
$f= {\rm Id} + \del \, h_{\e} \left(\cdot \,, \bigcup\limits_{i=1}^N (\partial E_i)_{\e,\cT}(t) \right)$  and by \eqref{mc_bound}, \eqref{jacobian_bound} and  \eqref{smallstep} one has  
$$ \max \{ \| f- {\rm Id} \|_{\infty} , \| Jf - 1 \|_{\infty} \} \leq  c_3 (M+1) \e^{-4} <  1,$$ this gives  
\begin{equation}
 |\cL^n(B(a,R) \cap ( E_i)_{\e,\cT}(t + \del))-\cL^n\left( B(a,R) \cap ( E_i)_{\e,\cT}(t)\right)|\leq c_{8} \e,
\end{equation}
for any $t\in [0, 1-\del]$ and this finishes the proof of  Lemma \ref{volumechange}. 
\end{proof}

\begin{theo}\label{thm:nontriviality}(Nontriviality of {\it limit} spacetime Brakke flows)

\noindent Let $\mathcal{E}=\bigcup\limits_{i=1}^N E_i$ be an open partition of $\R^n$ and let $\mu(t)$ be the mass measure of a limit spacetime Brakke flow starting from $\partial \mathcal{E}$. Then, there exists $t_0>0$ such that $\mu(t)(\R^n)>0, \forall t\in[0,t_0]$.
\end{theo}
\begin{proof}
Let $\mathcal{O}$ be one of the bounded open sets of $E$, denote by $\partial \mathcal{O}$ the natural varifold associated to its boundary. Let $(\e_j)_{j\in \N}$ be a sequence converging to $0$ such that if we denote the approximate MCF starting from $\partial E$ by $\left((\partial E)_{\e_j}(t)\right)_{t\in[0,1]}$ we have: 

\begin{equation*}
  (\partial E)_{\e_j}(t) \xrightarrow[j\rightarrow  \infty]{} \mu(t)(\R^n)\,\, \forall t\in[0,1].
\end{equation*}
Let $(\cT_j)_{j\in\N}$ a family of subdivisions of $[0,1]$, of uniform time step, satisfying \eqref{smallstep} and 
\begin{equation*}
 \del_j := \Big\lceil \e_j^{-n-12}\exp(c_6\e_j^{-n-7}) \Big\rceil^{-1}.
\end{equation*}
In the proof, we denote for simplicity  $(\partial E)_j(t):=(\partial E)_{\e_j,\cT_j}^{pw}(t)$ to be the piecewise approximate MCF starting from $\partial E$, as 
$$\del_j \leq \e_j^{n+12}\exp(-c_6\e_j^{-n-7}) = \e_j^{n+11}\exp(-c_6\e_j^{-n-7}) \, o(1),$$ 
we deduce from the discussion of the beginning section \ref{subsec:nontriviality} that 
\begin{equation}
 \|(\partial E)_j(t)\|(\R^n) \xrightarrow[j\rightarrow  \infty]{} \mu(t)(\R^n) \quad \forall t \in [0,1].
\end{equation}
The goal now is to prove that there exists $t_0 \in (0,1]$, a constant $\tilde{\omega} > 0$ such that $\mu(t_0)(\R^n)\geq \tilde{\omega}$, by the  decay property of the mass we obtain $ \mu(t)(\R^n) \geq \mu(t_0)(\R^n) \geq \tilde{\omega}$ for every $t\in[0,t_0]$ and this finishes the proof of the theorem.\\
We carry on with the proof, let $(a,R)\subset \R^n\times\R^+$ such that $B(a,R)\subset \mathcal{O}$.\\
Define $\psi(x,t)=\gamma(|x-a|^2+2dt)$ such that
\begin{equation}
\gamma(r)=\left\{
 \begin{array}{ll}
(R^2-r)^4 \quad  &\text{for} \,\, r\leq R^2, \\
0 \quad          &\text{for} \,\, r > R^2. 
 \end{array} \right.
\end{equation}

\noindent Let $(\partial \mathcal{O}_j(t))_{t \in [0,1]}$ denotes the evolution in time of the measure associated to $\mathcal{O}$ under the piece-wise approximate MCF of $\partial \mathcal{E}$. For $j$ large enough, we have $c_5 \del_j  \, \e_j^{-8}\leq \e_j $ and $\e_j\in(0,\e_0)$, hence from Lemma \ref{epsilonspherebarrier}
\begin{equation}
 \|(\partial E)_j(t)\|(\psi(\cdot,t))
 -\| (\partial E)_j(0)\|(\psi(\cdot,0))  \leq c_7 \e_j^{\frac16} \quad \forall  t\in[0,1].
\end{equation}
By construction of $\psi$,  we can assert that $\|(\partial  E)_j(0)\|(\psi(\cdot,0))=0$. We set $t_0=\frac{R^2}{8d}$, for any $t\in[0,t_0]$ and $x\in B:=B\left(a,\frac{R}{2}\right)$ we have
\begin{equation}
 \psi(x,t) = \phi(|x-a|^2 +2dt) \geq \left( R^2- \left( \frac{R^2}{4} +2dt_0 \right) \right)^4  =\frac{R^{8}}{2}
\end{equation}
This yields: $$\| (\partial E)_j(t)\|(\psi(\cdot,t)) \geq \frac{R^{8}}{2} \|(\partial E)_j(t)\|(B) \geq \frac{R^{8}}{2} \|\partial \mathcal{O}_j(t) \| (B) $$
and that $\|\partial \mathcal{O}_j(t) \| (B) \xrightarrow[j\rightarrow \infty]{}0$  uniformly on $[0,t_0]$.\\

\noindent We have the following two cases: 

\noindent ${\bf Case 1:}$ There exists a subsequence $\{ \varphi(j) \}_{j\in \N} \xrightarrow[j\rightarrow +\infty]{} +\infty $ such that $ \forall j\in \N ,\, \| \mathcal{O}_{\varphi(j)}(t_0) \| (B) \geq \frac14 \cL^n(B)$ then we can infer from the isoperimetric inequality that 
\begin{equation}
 \| \partial \mathcal{O}_{\varphi(j)}(t_0) \| (\R^n) \geq c_n (\|  \mathcal{O}_{\varphi(j)}(t_0) \| (\R^n))^{\frac{n-1}{n}} \geq c_n \|  \mathcal{O}_{\varphi(j)}(t_0) \| (B)^{\frac{n-1}{n}}  \geq c_n \left( \frac14 \cL^n(B) \right)^{\frac{n-1}{n}}:= \tilde{\omega}
\end{equation}
for some constant $c_n > 0$ depending only on $n$, taking $j$ to $+ \infty$ we deduce that $$\mu(t_0)(\R^n) = \lim\limits_j  \| (\partial E)_{\varphi(j)}(t_0) \| (\R^n) \geq \lim\limits_j  \| \partial \mathcal{O}_{\varphi(j)}(t_0) \| (\R^n) \geq  \tilde{\omega}$$ and  this finishes the proof.

\noindent ${\bf Case 2:}$ There is no such sequence, this implies that there exists $j_0 \in \N$ such that $\forall j \geq j_0$
$\| \mathcal{O}_j(t_0) \| (B) < \frac14\cL^n(B).$
We have $\| \mathcal{O}_j(0) \| (B) = \cL^n(B) $ and $\| \mathcal{O}_j(t_0) \| (B) < \frac14 \cL^n(B) $, by Lemma \ref{volumechange} we infer that for any $t\in [0,1-\del_j]$  
\begin{equation}
 \big| \| \mathcal{O}_j(t+\del_j) \| (B) -   \| \mathcal{O}_j(t) \| (B) \big| \leq c_{8} \e_j.
\end{equation}
Then, taking $j_0$ bigger so that $c_{8} \e_j \leq \frac14 \cL^n(B)$  we can infer that  there exists $s\in[0,t_0]$ such that 
$$ \frac12\cL^n(B) \geq \| \mathcal{O}_j(s) \| (B) \geq \frac14\cL^n(B)$$ 
for all $j \geq j_0$. By the relative isoperimetric inequality (\cite[Remark 3.50]{afp}), there  exists a constant $\tilde{c}_n>0$ (depending only on $n$) such that  for any $j \geq j_0$ 
\begin{equation}
 \| \partial \mathcal{O}_j(s) 
 \| (B) \geq \tilde{c}_n \min \lbrace \cL^n (B \cap \mathcal{O}_j(s))^{\frac{n-1}{n}}, \left( \cL^n (B \setminus \mathcal{O}_j(s) \right)^{\frac{n-1}{n}} \rbrace \geq \tilde{c}_n\left( \frac{1}{4} \cL^n(B)\right)^{\frac{n-1}{n}}
\end{equation}
this yields a contradiction, as $\|\partial \mathcal{O}_j(s) \| (B) \xrightarrow[j\rightarrow +\infty]{} 0 $  .\\
To sum up, there exists $\tilde{\omega} >0$ such that $\mu(t)(\R^n) \geq \tilde{\omega} > 0\, \forall t \in [0,t_0]$ and we finish the proof of the theorem.\\
\end{proof}
\begin{remk}[Quantification of nontriviality time interval] 
From the proof of Theorem \ref{thm:nontriviality} we can see that the flow is nontrivial for $t\in[0,t_0)$ where
$$ t_0 = \max \{ R > 0, \, B(a,R) \subset E_j \in \mathcal{E},  \, \text{for some} \,a \in \R^n  \, \text{and} \,  E_j \,\, \text{bounded}  \} $$    
\end{remk}

\subsection{Avoidance principle for codimension $1$ spacetime Brakke flows, inclusion in level set flows}

In this section we prove that the mass measure of a codimension $1$ spacetime Brakke flow avoids smooth codimension $1$ mean curvature flows. The proof is a slight adaptation of the proof of \cite[10.5]{Ilmanen}. We start by showing the lower semi-continuity property of the map $t \mapsto \mu(t)$.
\begin{prop}[Lower semi continuity]\label{prop:lower_semi_continuity}
Let $V_0 \in V_d(\R^n)$ of compact support, $\left( \mu(t) \right)_{t \in [0,1]}$ be the mass measure of its spacetime Brakke flow (Definition \ref{def:spacetimebf}). Then for any $\psi \in C^2_c(\R^n , \R^+)$,
 \begin{itemize}
  \item[$(i)$] The map $t \mapsto \mu(t)(\psi) - Ct$ is nonincreasing for any $C \geq \|\nabla^2\psi\|_{\infty}||V_0||(\R^n)$.
 \end{itemize}

\begin{itemize}
 \item[$(ii)$] For any $\phi \in C^2_c(\R^n \times [0,1] , \R^+)$ and any  $s\in (0,1]$,
 \begin{equation}\label{eq:lower_semi_continuity}
  \lim_{t \rightarrow s^-}\mu(t)(\phi(\cdot,t)) \geq \mu(s)(\phi(\cdot,s)).
 \end{equation}
\end{itemize}
\end{prop}
\begin{proof}
Let $\psi  \in C^2_c(\R^n , \R^+)$, from the integral Brakke inequality \eqref{intbrakkeineq}, and using $ab \leq \frac12(a^2+b^2)$ we obtain (for simplicity we denote $h=h(\cdot,\cdot,\lambda$) )
\begin{equation}
\begin{split}
\mu(s)(\psi) - \mu(r)(\psi) 
&\leq - \int_r^s \int_{\R^n} \psi |h|^2 \, d\mu(t)dt + \int_r^s \int_{\R^n}S ^{\perp}(\nabla\psi)\cdot h \, d\lambda \\& \leq \int_r^s \int_{\R^n} - \psi |h|^2 + \frac12 \psi|h|^2 + \frac12 \frac{|\nabla\psi|^2}{\psi} d\mu(t)dt 
 \\& \leq (s-r) \|\nabla^2\psi\|_{\infty}||V_0||(\R^n) \quad \text{by \cite[Lemma 3.1]{ton}}
 \end{split}
\end{equation}
where we used $\mu(t)(\R^n) \leq \|V_0\|(\R^n)$, this proves $(i)$.

\noindent Let $s \in[0,1]$, we first prove $(ii)$ for time-independent test functions. Let $\phi \in C^2_c(\R^n \times [0,1] , \R^+)$, for $\psi = \phi (\cdot,s)$ and $C:= \|\nabla^2\psi\|_{\infty}||V_0||(\R^n)$
we know from $(i)$ that the map $t \mapsto \| \mu(t) \|(\psi) - Ct$ is nonincreasing. Thus,  $\lim\limits_{t\rightarrow s^- } \left( \mu(t)(\psi) - Ct \right) \geq \mu(s)(\psi) - Cs$ which yields $\lim\limits_{t\rightarrow s^- } \mu(t)(\phi(\cdot,s)) \geq \mu(s)(\phi(\cdot,s))$.
\noindent Then $(ii)$ follows from
\begin{equation}
\big| \lim\limits_{t\rightarrow s } \mu(t)(\phi(\cdot,s) - \phi(\cdot,t)) 
\big| \leq \lim\limits_{t\rightarrow s }  (s-t)\| \phi \|_{C^1} \| V_0\|(\R^n) =0. 
\end{equation}
\end{proof}
In the following lemma we prove a continuity property of Brakke flows and the mass measures of spacetime Brakke flows (seen as a subset of $\R^n$).
\begin{lemma}[A continuity property]\label{avoidance_continuity}
Let $\mu(t)$ be the mass measure of a  spacetime  Brakke flow 
(Definition \ref{def:spacetimebf}) and starting from $V_0 \in    V_d(\R^n)$, or a mass measure of a Brakke flow. Then, for any $r> 0$ and a set $A$  of $\R^n$, we have
\begin{equation}
 \supp \mu(t_0) \cap \left(  A+B_r \right) = \emptyset \, \implies \,  \supp\mu(t_0+t) \cap \left( A+ B_{\sqrt{r^2-2dt}} \right)= \emptyset \, \forall \, t\in \big[0,r^2/2d\big].
\end{equation}
\end{lemma}
\begin{proof}
Let $\mu(t)$ be defined as in the lemma, let $t_0 \in [0,1]$ and $A$: a set of $\R^n$ such that $ \supp \mu(t_0) \cap \left(A+B_r \right) = \emptyset$. We write $A+B_r = \cup_{x\in A} B_r(x)$ hence $ \supp \mu(t_0) \cap B_r(x) = \emptyset \, \forall x \in A$.  By the avoidance principle for external varifolds (Proposition \ref{externalvar}) we can infer that $$\supp\mu(t_0+t) \cap  B_{\sqrt{r^2-2dt}}(x) = \emptyset \, \forall t \in [0,r^2/2d] \, \forall x \in A,$$ 
the result follows from noting that $ A+ B_{\sqrt{r^2-2dt}} = \bigcup\limits_{x\in A} B_{\sqrt{r^2-2dt}}(x)$. The result is valid for Brakke flows and mean curvature flows as it is based only on the avoidance principle for external varifolds (which is true for Brakke flows and MCFs by \cite[Theorem 3.7]{brakke}).
\end{proof}

We now show that the mass measures of codimension $1$ spacetime Brakke flows avoids smooth codimension $1$ mean curvature flows.
\begin{theo}[Avoidance of smooth MCFs]\label{avoidance}
Let $V_0\in V_{n-1}(\R^n)$, let $(\mu(t))_{t\in[0,T]}$ be the mass measure of a spacetime Brakke flow starting from $V_0$ (see Definition \ref{def:spacetimebf}). Let $(M_t)_{t\in[0,T]}$ be the MCF of a compact $C^2$ hypersurface $M$. We have:
$$ \supp\mu(0) \cap \supp M_0 = \emptyset \implies \supp\mu(t) \cap \supp M_t = \emptyset \quad \forall \, t\in[0,T].$$
\end{theo}

\begin{proof}

The idea is to construct a test function $\psi(\cdot,t)$ out of the distance function to $M_t$ vanishing outside a neighborhood of $M_t$ and prove that  $\mu(t)(\psi(\cdot,t))=0 \, \forall \, t \in [0,T]$. We assume that $\supp \mu(0) \cap M_0 = \emptyset $.

\noindent {\bf{Step1:}} (Construction, properties of the test function).  Let $E_t$ be the compact region bounded by $M_t$, define 
\begin{equation}
r(x,t)=\left\{
 \begin{array}{ll}
-\dist(x,M_t) , \quad x \in E_t, \\
\dist(x,M_t), \quad x \in E_t^c.
 \end{array} \right.
\end{equation}
Fix $\gamma>0$ so small that $$\dist(\supp \mu(0), M_0) > \gamma, $$
and such that $r(x,t)$ is smooth on a set $U+B_{\e}$, $\e > 0$ small, where
$$ U \equiv \lbrace (x,t) : -\gamma < r(x,t) < \gamma, 0 \leq t \leq T\rbrace.$$ 
this is possible due to the compactness of $[0,T]$. Let $\beta > 0$ such that 
\begin{equation}\label{avoidance_def_beta}
 \beta > \max \big\{ 2 \,,\, \frac34(1+\gamma \max_{\overline{U}}\| \nabla^2 r \|) \big\}.
\end{equation}
Define the test function 
\begin{equation}
\psi=\phi(r)=\left\{
 \begin{array}{ll}
(\gamma -|r| )^{\beta} &\quad |r| \leq \gamma \\
\,\,0 &\quad |r| \geq \gamma.
 \end{array} \right.
\end{equation}
Then $\psi \in C_c^0(\R^n \times [0,T], \R^+)$, $\psi$ vanishes except on $U$, and $\psi$ is $C^2$ except along $M_t$.\\
Observe that $\mu(0)(\psi(\cdot,0))=0$. We aim to show this remains so for $t \in [0,T]$ and this will complete the proof. We derive the expressions of the first two derivatives of $\psi$ and  $\phi$ and a property between those of $\phi$  that will be used later in the proof. We have, 
\begin{equation}\label{avoidance_def_psi}
 \nabla \psi = \phi'(r)\nabla r,  \quad \nabla^2\psi = \phi''(r)\nabla r \otimes \nabla r + \phi'(r) \nabla^2 r, \quad \partial_t \phi = \phi'(r) \partial_t r,
\end{equation}
and
\begin{equation}\label{avoidance_def_phi}
\begin{split}
  &\phi = (\gamma - r)^{\beta}, \\
  &\phi ' = -\text{sign}(r) \beta (\gamma -|r|)^{\beta}, \\ 
  &\phi'' = \beta (\beta-1)(\gamma-|r|)^{\beta-2},\\
  &\frac{(\phi')^2}{\phi} = \beta^2 (\gamma-|r|)^{\beta-2}.
\end{split}
\end{equation}

\noindent $\bf{Step 2:}$ We show that if $\mu(s)(\psi(\cdot,s))=0, \, s\in [0,T)$ then there exists $\tau>0$ such that $\supp \mu(t) \cap M_t = \emptyset, \, \, \forall t \in [s,s+\tau]$.

\noindent Indeed, $\mu(s)(\psi(\cdot,s))=0$ for some $ s\in [0,T)$ implies that $\dist(\supp \mu(s), M_s) \geq \gamma$. Assume without loss of generality that $\supp \mu(s)$ lies outside $E_s$ (the compact region bounded by $M_s$). Applying Lemma  \ref{avoidance_continuity} once with $A=E_s+B_{\gamma/2}, r= \gamma/2$ and $(\mu(t))_{t\in[0,T]}$ and a second time with $A= \left( E_s+B_{\gamma/2}\right)^c, r= \gamma/2$ and  $(M_s)_{t\in[0,T]}$ yields that:
\begin{equation*}
  \exists \tau > 0, \, \supp \mu(t) \cap (E_s+B_{\gamma/2}) = \emptyset \,\, \text{and} \, M_t \cap (E_s+B_{\gamma/2})^c = \emptyset \,  \forall t \in [s,s+\tau].
\end{equation*}
thus $\supp \mu(t) \cap M_t = \emptyset, \, \, \forall t \in [s,s+\tau]$, this concludes step 2.

\noindent $\bf{Step 3:}$ We carry on with the proof of the theorem. Let $s=\sup \lbrace t \in [0,T] : \mu(t)(\psi(\cdot,t))=0 \rbrace$, from Proposition \ref{prop:lower_semi_continuity} we infer that $\mu(s)(\psi(\cdot,s))=0$. In case $s=T$, we obtain that $\supp \mu(T) \cap M_T = \emptyset$, the theorem follows directly as $T$ can be replaced by any arbitrary real number in $[0,T]$. Assume by contradiction that $s < T$. Step $2$ implies: 
\begin{equation}\label{avoidance_def_tau}
 \exists \tau > 0, \, \supp \mu(t) \cap M_t = \emptyset, \,\, \forall t \in [s,s+\tau].
\end{equation}
We will prove that $\mu(s+\tau)(\psi(\cdot,s+\tau))=0$ which yields a contradiction. \eqref{avoidance_def_tau} implies that $\psi(\cdot,t)$ is smooth on $\supp \mu(t), \, \forall \, t \in [s,s + \tau]$, we have from \ref{intbrakkeineq}
\begin{equation}
\begin{split}
 \mu(s+\tau)&(\psi(\cdot,s+\tau)) - \mu(s)(\psi(\cdot,s))  \leq -\int_{s}^{s+\tau}\int_{\R^n} \psi(y,t)|h(y,t,\lambda)|^2 \, d\mu(t)(y)dt
\\& +\int_{s}^{s+\tau}\int_{\R^n \times \G } S^{\perp}(\nabla\psi(y,t))\cdot h(y,t,\lambda) \, d\lambda(y,S,t) + \int_{s}^{s+\tau}\int_{\R^n} \partial_t \psi(y,t)\, \mu(t)(y)dt.
\end{split}
\end{equation}
By the definition of $s$, Lemma \ref{technical} we have (by abuse of notation we identify $f(y,S,t)$ to $f(y,t)$ for $f=\psi(y,t)|h(y,t,\lambda)|^2$ and $f=\partial_t \psi(y,t)$ so that the integrals make sense)
\begin{equation}
\begin{split}
 \mu(s+\tau)(\psi(\cdot,s+\tau)) &\leq \int_{s}^{s+\tau}\int_{\R^n\times \G} \frac14\frac{|S(\nabla\psi(y,t))|^2}{\psi(y,t)} + \nabla\psi(y,t)\cdot h(y,t,\lambda)  + \partial_t\psi(y,t) \, d\lambda(y,S,t).
\end{split}
\end{equation}
We know that $\psi(\cdot,t) $ is $C^2$ on $\supp \mu(t)$ for every $t\in [s, s+\tau]$, we have by \eqref{st_firstvar}
\begin{equation*}
 \int_{s}^{s+\tau}\int_{\R^n} \nabla\psi(y,t) \cdot h(y,t,\lambda) d\mu(t)dt = \int_{s}^{s+\tau}\int_{\R^n\times \G} - S:\nabla^2\psi(y,t) d\lambda(y,S,t) 
\end{equation*}
It yields, 
\begin{equation}\label{negativeintegrand}
\begin{split}
 \mu(s+\tau)(\psi(\cdot,s+\tau)) &\leq \int_{s}^{s+\tau}\int_{\R^n\times \G} \frac14\frac{|S(\nabla\psi(y,t))|^2}{\psi(y,t)} - S:\nabla^2\psi(y,t)  + \partial_t \psi(y,t) \, d\lambda(y,S,t).
\end{split}
\end{equation}

\noindent We now prove that the integrand of \eqref{negativeintegrand}, that we denote $I$, is nonegative for all $(x,S,t) \in \R^n \times \G \times [s,s+\tau]$ (in fact, the integrand is $\leq 0$ whenever $\psi$ is differentiable).

\noindent Plugging $\psi=\phi(r)$ into $I$, and using \eqref{avoidance_def_psi} we get
\begin{equation*}
 I = \left( \frac14 \frac{\phi'^2}{\phi}-\phi''\right)|S(\nabla r)|^2  + \phi' \left( -S:\nabla^2r+ \partial_t r \right) 
\end{equation*}
Since $M_t$ is a smooth mean curvature flow, a standard calculation (see for instance \cite[\text{Identity $(6.4)$}]{evsp2}) tells us that 
\begin{equation*}
 r \left( \partial_t r - \Delta r \right) \geq 0,
\end{equation*}
on $U$, where $\Delta = \Delta^{\R^n}$.  Since $\phi'(r) r \leq 0$, we have
\begin{equation*}
 \phi'(r) \left(\partial_t r - \Delta r \right) \leq 0.
\end{equation*}
Now $ 1 = |\nabla r |^2 = |S(\nabla r)|^2+ |\nabla r(\vec n)|^2$ and $\Delta r = S:\nabla^2 r +  \nabla^2r (\vec n)\cdot \vec{n}$, where $\vec{n}$ is the unit normal to the hyperplane $S$, and therefore 
\begin{equation}\label{avoidance_prf9}
 I \leq \left( \frac14 \frac{\phi'^2}{\phi}-\phi''\right)(1-|\nabla r ( \vec{n})|^2) + \phi' \nabla^2r (\vec{n}) \cdot \vec{n}.
\end{equation}
For $x\in U$, define the hyperplane $T(x)=\nabla r^{\perp}(x)$, note that $\vec{n}=T(\vec{n}) + \left( \vec{n}  \cdot \nabla r \right) \nabla r$, thus 
\begin{equation}\label{avoidance_prf10}
 1 = |\nabla r |^2 = |T(\vec{n})|^2 + | \nabla r(\vec{n})|^2,
\end{equation}
the identity $|\nabla r | =1$ yields $ \nabla^2 r (\nabla r)= \nabla |\nabla r |^2 =0$ and 
\begin{equation}\label{avoidance_prf11}
\begin{split}
 \nabla^2 r (\vec n )\cdot \vec n & = \nabla^2 r \left( T(\vec n ) , T (\vec n ) \right) + 2 \nabla^2 r \left( \nabla r (\vec n ),  T (\vec n ) \right) + \nabla^2 r \left( \nabla r (\vec n), \nabla r (\vec n) \right) 
 \\& = \nabla^2 r \left( T(\vec n ) , T (\vec n ) \right) \leq ||\nabla^2 r|||T(\vec n) |^2
 \end{split}
\end{equation}
Injecting \eqref{avoidance_prf10} and \eqref{avoidance_prf11} into \eqref{avoidance_prf9}  we obtain 
\begin{equation} \label{avoidance_prf12}
 I \leq \left(  \frac14 \frac{ \phi'^2}{\phi}-\phi''+ |\phi'| ||\nabla^2 r || \right) |T(\vec{n})|^2 \leq \left(  \frac14 \frac{ \phi'^2}{\phi}-\phi''+ |\phi'| ||\nabla^2 r || \right).
\end{equation}
Substituting the computations of \eqref{avoidance_def_phi} into \eqref{avoidance_prf12} we obtain
\begin{equation}
\begin{split}
 I &\leq \beta (\gamma-|r|)^{\beta-2}\left( 1 - \frac34 \beta + (\gamma -|r|) \| \nabla^2 r \|\right) 
\\& \leq \beta (\gamma-|r|)^{\beta-2}\left( 1 - \frac34 \beta + \gamma  \| \nabla^2 r \|\right). 
 \end{split}
 \end{equation}
We obtain by the choice of $\beta$ that  $I \leq 0$, therefore from \eqref{negativeintegrand} we have  $\mu(s+\tau)(\psi(\cdot,s+\tau))=0$ and this contradicts the definition of $s$, thus $s=T$ and we are done.
\end{proof}

\begin{remk}
We notice from the proof of Theorem \ref{avoidance} that the distance between the mass measure of a spacetime Brakke flow and a smooth flow is nondecreasing. We will see later in Corollary \ref{cor:avoidance} that more generally, the distance between the masses of two spacetime Brakke flows is nondecreasing.
\end{remk}

%

%
%
%
According to \cite[Definition 10.1]{Ilmanen}, the mass measure of a spacetime Brakke flow is a {\it set-theoretic subsolution} of the mean curvature flow. This has some major consequences that we  list in the next corollary:

\begin{cor}\label{cor:avoidance}
Let $(\mu(t))_{t\in[0,1]}$ the mass measure of a spacetime Brakke flow starting from $V_0\in V_{n-1}(\R^n)$ assumed to be bounded (Definition \ref{def:spacetimebf}). Let $(\Gamma_t)_{t\in[0,1]}$ be a level set flow such that $\supp \mu(0) \subset \Gamma_0$, then

\noindent \begin{enumerate}
\item(Inclusion in level set flows) 
\begin{equation}
\supp \mu(t) \subset \Gamma_t \,\, \forall \, t \in [0,1].
\end{equation}

\item(Smooth level set case) Assume that $\Gamma_0$ is smooth and that $\supp\mu(0) \neq \Gamma_0$,  this implies that  $\supp \mu(t) = \emptyset,\, \forall t > 0$. In particular, if $\mu(t)(\R^n) > 0$ on $[0,s],\, s \in [0,1]$ then $\supp \mu(t) = \Gamma_t, \forall t \in [0,s)$ by $(1)$.

\item(Avoidance for spacetime Brakke flows) Let $(\mu_1(t))_{t\in[0,1]}$ be the mass measure of a spacetime Brakke flow, then 
\begin{equation}
 \supp \mu(0) \cap \supp \mu_1(0) = \emptyset \quad \implies \quad 
 \dist(\supp \mu(t), \supp \mu_1(t)) \,\, \text{is nondecreasing}. 
\end{equation} 

\end{enumerate}
\end{cor}
\begin{proof}
\noindent
\begin{enumerate}

\item \cite[Inclusion Theorem 10.7]{Ilmanen}.

\item Assume that $\supp\mu(0) \neq \Gamma_0$, as $\supp\mu(0)$ is closed in $\Gamma_0$, there exists an open set of $\Gamma_0$ that we call $o$ satisfying $\supp \mu(0) \subset \Gamma_0 \setminus o$. Perturbing $o$ smoothly, one can construct $\Gamma'_0$ smooth containing $\supp \mu(0)$ and contained in the closure of the domain bounded by $\Gamma_0$.  By \cite[Theorem 4.1]{evsp2} we know that the level set flows of $\Gamma_0$ and $\Gamma'_0$ split instantaneously, the result follows directly as $\supp(t)$ is included in both level set flows.

\item Consequence of \cite[10.1]{Ilmanen}.
\end{enumerate}

\end{proof}

\section{Weak compactness of spacetime Brakke flows}
The set of spacetime Brakke flows is weakly compact, i.e. any bounded sequence converges to a spacetime Brakke flow. The proof is an adaptation of the proof of \cite[Theorem3.7]{ton}. One can also draw an analogy with the proof of the convergence of the approximate Brakke flows to Brakke flows (see \cite{brakke}, \cite{kt}).\\

\noindent Before stating the compactness theorem we start by stating and proving the following key lemma (which seems to be standard but we could not find it in any reference).
\begin{lemma}\label{cvmonotone}
Let $\lbrace f_j \rbrace_{j \in \N}$ be a family of nonincreasing functions on $\R$. Assume that there exists a dense set $\mathcal{A}\subset \R $, such that: for all $t\in\mathcal{A}$,  $( f_j(t))_{j \in \N}$ converges. Then, $(f_j(t))_{j \in \N}$ converges for a set  $\mathcal{S}$ of countable complementary, and the limit function is nonincreasing on $\mathcal{S}$.
\end{lemma}
\begin{proof}
Let $D_j$  be the set of discontinuity of $f_j$, $D_j$ is countable as $f_j$ is monotone. We set  $\displaystyle D=\bigcup\limits_{j\in\N} D_j$, $D$ is countable too. We will prove that $(f_j(t))_{j \in \N}$ converges for all $t\in \R \setminus D$.\\
Choose $t\in \R \setminus D$, choose two real sequences $(a_j)_{j \in \N}$, $(b_j)_{j \in \N}$  in $\mathcal{A}$ such that $\lim\limits_j a_j = \lim\limits_j b_j = t$ and $a_j\leq t \leq b_j$ $\forall j\in \N$, and
\begin{equation}
\big| f_j(a_j) - f_j(b_j) \big| \leq \frac{1}{j},
\end{equation}
this is possible as the function $f_j$ is continuous at $t$ for all $j\in\N$. We have : 

\begin{equation}
\lim\limits_j f_j(a_j) = \lim\limits_j f_j(b_j)
\end{equation} 
and for all $j\in \N$
\begin{equation}
f_j(a_j) \geq f_j(t) \geq f_j(b_j)
\end{equation} 
Thus $\lim\limits_j f_j(t)$ exists and is equal to $\lim\limits_j f_j(a_j) (=\lim\limits_j f_j(b_j))$. The fact that $f$ is nonincreasing on $\R \setminus D$ is a direct result, we set $\mathcal{S}=\R \setminus D$, this finishes the proof.
\end{proof}

We now state and prove the compactness property of spacetime Brakke flows.
\begin{theo}(Weak compactness)\label{thm:weak_compactness}
Let $\left( \lambda_j \right)_{j\in \N}$ be a sequence of spacetime Brakke flows (Definition \ref{def:spacetimebf}) starting from a sequence of varifolds $\left( V_{j} \right)_{j\in \N}$. Let $M\geq 0$, assume that $\|V_j\| (\R^n) \leq M, \, \forall j \in \N$ and that  $$ \forall j \in \N, \, \supp  \| V_{j}\| \subset K \subset \R^n, \text{where $K$ is a compact set.} $$
Then, there exists a subsequence $j' \in \N$ and a Radon measure $\lambda$, such that,
\begin{itemize}
 \item[$(i)$] $\lambda = \lim\limits_{j'} \lambda_{j'}$.
 \item[$(ii)$] $\lambda$ is a spacetime Brakke flow. 
\end{itemize}
\end{theo}

\begin{proof}
For all $j \in \N$, $\lambda_j$ is a spacetime Brakke flow, thus, $\lambda_j = \mu_j(t) \otimes (\nu_j)_{(x,t)} \otimes dt$ where $\mu_j(t)$ is a Radon measure on $\R^n$ for all $t\in[0,1]$ and $(\nu_j)_{(x,t)}$ is a probability measure on $\G$ for all $(x,t) \in \R^n \times [0,1]$.

\noindent {\bf{Step1:}} We show that there exists a subsequence $j'$ such that: 
\begin{equation}
 \lim_{j'} \lambda_{j'} = \lambda=\mu(t) \otimes \nu_{(x,t)} \otimes dt \,\, \text{and}  \,\, \lim\limits_{j'} \mu_{j'}(t) = \mu(t), \, \forall t \in [0,1]
\end{equation}
where $\mu(t)$ is a Radon measure on $\R^n, \, \forall t \in[0,1]$ and $\nu_{(x,t)}$ is a probability measure on $\G, \, \forall (x,t) \in \R^n \times [0,1]$. 

\noindent We assert from the mass decay property (Remark \ref{remk:spacetimebf} $(ii)$) that
$\forall j \in \N, \, \mu_j(t)(\R^n) \leq \mu_j(0)(\R^n) = \|V_j\|(\R^n) \leq M$, therefore, there exists a subsequence $j'$ and a family of Radon measures  $(\mu(t))_{t\in[0,1]}$ such that $\mu_{j'}(t) \xrightharpoonup{*} \mu(t)$ for $t$ in a (countable) dense set $\mathcal{A} \subset [0,1]$. Let $Z = \{ \psi_q \}_q \in C_c^2(\R^n,\R^+)$ be a dense subset of $C_c(\R^,\R^+)$. From Proposition \ref{prop:lower_semi_continuity} $(i)$, the map $g_{q,j'}(t) = \mu_{j'}(t)(\psi_q) - C_{q}t $ is nonincreasing where $C_q = \|\nabla^2\psi_q\|_{\infty} M$, Lemma \ref{cvmonotone} implies that $g_{q,j'}(t)$ converges (as $j'\rightarrow \infty$) on a set $\mathcal{S} \subset [0,1]$  of countable complementary, thus $\mu_{j'}(t)$ converges for all $t\in \mathcal{S}$  to a measure that we still denote by $\mu(t)$. By an additional extraction of the sequence $j'$ (and keeping the same notation) we can guarantee that $\mu_{j'}$ converges to a measure that we (still) denote $\mu(t)$ for all $t\in [0,1]$.

\noindent Finally, we know that 
$$ \lambda_{j'}(\R^n \times \G \times [0,1]) = \int_0^1 \mu_{j'}(t) (\R^n)\, dt \leq M.$$ 
Hence, by a further extraction of $j'$ (keeping the notation $j'$ for the extracted sequence) we can assume that $\lambda_{j'} \xrightharpoonup{*} \lambda$, for a certain Radon measure $\lambda$ on $\R^n \times \G \times [0,1]$.

\noindent We now show the decomposition property for the measure $\lambda$. Indeed, we have $\mu_{j'}(t) \otimes dt \xrightharpoonup{*} \mu(t) \otimes dt $ (in general if $a_j$ is a sequence of measures converging to $a$ and $b$ is a bounded measure then, $a_j\otimes b \xrightharpoonup{*} a \otimes b)$, moreover if $\Pi : \R^n \times \G \times [0,1] \to \R^n \times [0,1]$ denotes the canonical projection, we have:
\begin{equation*}
 \Pi_{\#} \lambda_{j'} \xrightharpoonup{*} \Pi_{\#}\lambda
\end{equation*}
we know that $\Pi_{\#} \lambda_{j'} = \mu_{j'}(t) \otimes dt$, thus $\Pi_{\#}\lambda = \mu(t) \otimes dt$. It follows by Young's disintegration theorem ~\cite[Theorem 2.28]{afp} that there exists a family of probability measures on $\G$, that we denote $\lbrace \nu_{(x,t)} \rbrace_{(x,t)} $, such that: 
\begin{equation}
 \lambda = \mu(t)\otimes \nu_{(x,t)} \otimes dt, 
\end{equation}
which completes the proof of step $1$.
%

\noindent {\bf{Step 2:}} We prove that $\delta \lambda$ is bounded, $(\delta\lambda)_s=0$, $h(\cdot,\cdot,\lambda) \in L^2(\R^n \times [0,1], \R^n , \mu(t)\otimes dt)$ and that 
\begin{equation*}
 \int_{0}^1 \int_{\R^n} |h(x,t,\lambda)|^2 \, d\mu(t) dt  \leq \mu(0)(\R^n).
\end{equation*}

\noindent Let $\phi \in C^1_c(\R^n \times[0,1], \R^+)$ and $X \in C_c^1(\R^n \times [0,1], \R^n)$. By the varifold convergence, and the Cauchy-Schwarz inequality we infer that (denote for simplicity $h_{j'} :=h_{j'}(\cdot,\cdot,\lambda_{j})$)
\begin{equation}\label{delta_lambda_bound}
\begin{split}
 \delta \lambda(\phi X ) =  \lim\limits_{j'} \delta \lambda_{j'} (\phi X) 
 &= \lim\limits_{j'} \int_{\R^n \times [0,1]} \phi X \cdot h_{j'} \, d\mu_{j'}(t)dt 
\\& \leq  \varliminf_{j'} \left( \int_{\R^n \times [0,1]} \phi |X|^2d\mu_{j'}(t)dt  \right)^{\frac12}   \left( \int_{\R^n \times [0,1]} \phi |h_{j'}|^2 \, d\mu_{j'}(t)dt \right)^{\frac12}
\\& =  \left( \int_{\R^n \times [0,1]} \phi |X|^2d\mu(t)dt  \right)^{\frac12}  \varliminf_{j'}  \left( \int_{\R^n \times [0,1]} \phi |h_{j'}|^2 \, d\mu_{j'}(t)dt \right)^{\frac12}
 \end{split}
\end{equation}
The measure $\lambda_{j}$ is a spacetime Brakke flow for any $j\in \N$, from Remark  \ref{remk:spacetimebf} $(iii)$ and for $\phi \equiv 1$ we have,
\begin{equation*}
 \delta \lambda (X) \leq \| X\|_{L^2(\mu(t)\otimes dt)}  \varliminf_{j'}  \left(\mu_{j'}(0)(\R^n)\right)^{\frac12}=  \| X\|_{L^2(\mu(t)\otimes dt)} \left(\mu(0)(\R^n)\right)^{\frac12}
\end{equation*}
for $X \in C_c^1(\R^n \times [0,1], \R^n)$. The space $C_c^1(\R^n \times [0,1], \R^n)$  is dense in $L^2(\R^n \times [0,1], \R^n)$.\\
Therefore, by Riesz representation theorem $ \exists \hat{h} \in L^2(\R^n \times [0,1], \R^n, \mu(t) \otimes dt) $ such that
 \begin{equation*}
 \| \hat{h}\|_{L^2(d\mu(t)\otimes dt)} \leq \mu(0)(\R^n) \quad \text{and} \quad \delta\lambda(X) = \int_{0}^1 \int_{\R^n} X \cdot \hat{h}\, d\mu(t) dt
\end{equation*}
for all $X \in L^2(\R^n \times [0,1], \R^n, \mu(t)\otimes dt)$. By the inclusion of $C^0(\R^n \times [0,1], \R^n, \| \cdot \|_{\infty})$ in $L^2(\mu(t)\otimes dt)$, this is due to the finiteness of the measure $\mu(t)\otimes dt$, we deduce that $\lambda$ has a bounded first variation, $h(\cdot,\cdot,\lambda) = - \hat{h}$ and $(\delta \lambda)_s =0$. Moreover,
\begin{equation*}
 \| h(\cdot,\cdot,\lambda)\|_{L^2(d\mu(t)\otimes dt)} \leq \mu(0)(\R^n) 
\end{equation*}
this finishes the proof of step $2$.

\noindent {\bf{Step 3:}} We prove that $\lambda$ satisfies the integral Brakke inequality \eqref{intbrakkeineq}. 

\noindent From now on, we denote for simplicity $h := h(\cdot,\cdot,\lambda)$ and $h_{j'} := h(\cdot,\cdot,\lambda_{j'})$. The proof of step $3$ consists of taking the $\varlimsup\limits_{j'}$ of the terms of the integral Brakke inequality satisfied by $\lambda_{j'}$, which is

\begin{equation}\label{stbrakkeinequalities}
  \begin{split}
 \mu_{j'}(t_2)(\phi(\cdot,t_2)) - &\mu_{j'}(t_1)(\phi(\cdot,t_1))  \leq -\int_{t_1}^{t_2}\int_{\R^n} \phi |h_{j'}|^2 \, d\mu_{j'}(t)dt
\\& +\int_{t_1}^{t_2}\int_{\R^n \times \G } S^{\perp}(\nabla\phi)\cdot h_{j'} \, d\lambda_{j'} + \int_{t_1}^{t_2}\int_{\R^n} \partial_t \phi\, d\mu_{j'}(t)dt.
\end{split}
 \end{equation}
for any $ 0 \leq t_1 \leq t_2 \leq 1$ and $\phi \in C^1_c(\R^n \times [0,1], \R^+)$.

\noindent Let $\phi \in C^1_c (\R^n \times [0,1], \R^+)$ and $t_1,t_2 \in [0,1]$ such that $0 \leq t_1 \leq t_2 \leq 1$, by the varifold convergence (step $1$), we have 
\begin{equation}\label{compactness_stbf0}
 \varlimsup\limits_{j'} \mu_{j'}(t)(\phi(\cdot,t)) = \mu(t)(\phi(\cdot,t), \, \forall t \in [0,1] \,\, \text{and}
 \,\, \varlimsup\limits_{j'} \int_{t_1}^{t_2}\int_{\R^n} \partial_t \phi \, d\mu_{j'}(t)dt = \int_{t_1}^{t_2}\int_{\R^n} \partial_t \phi \, d\mu(t)dt,
\end{equation}
we are now left with  the two terms involving the mean curvature.

\noindent We now show that 
\begin{equation}\label{compactness_stbf1}
  \varlimsup_{j'} - \int_{\R^n \times [t_1,t_2]} \phi |h_{j'}|^2 \, d\mu_{j'}(t)dt \leq - \int_{\R^n\times [t_1,t_2]} \phi |h|^2 \, d\mu(t) dt 
\end{equation}
Let $(X_q)_q \in C_c^1(\R^n \times [t_1,t_2], \R^n)$ converging in $L^2$ to $h_{|[t_1,t_2]}$ as $q \rightarrow \infty$, similar computations to \eqref{delta_lambda_bound} induce
\begin{equation*}
\begin{split}
 \int_{\R^n\times [t_1,t_2]} \phi X_q \cdot h \, d\mu(t) dt 
 &= - \delta \lambda(\phi X_q) =  - \lim\limits_{j'} \delta \lambda_{j'}(\phi X_q)=\lim\limits_{j'} \int_{\R^n\times [t_1,t_2]} \phi X_q\cdot h_{j'} \, d\mu_{j'}(t)dt 
 \\&  \leq  \left( \int_{\R^n \times [t_1,t_2]} \phi |X_q|^2d\mu(t)dt  \right)^{\frac12}  \varliminf_{j'} \left( \int_{\R^n \times [t_1,t_2]} \phi |h_{j'}|^2 \, d\mu_{j'}(t)dt \right)^{\frac12}
\end{split}
\end{equation*}
therefore, letting $ q \rightarrow \infty$ we obtain 
\begin{equation*}
  \int_{\R^n\times [t_1,t_2]} \phi |h|^2 \, d\mu(t) dt \leq   \varliminf_{j'}  \int_{\R^n \times [t_1,t_2]} \phi |h_{j'}|^2 \, d\mu_{j'}(t)dt  
\end{equation*}
this proves \eqref{compactness_stbf1}.

\noindent We now show that 
\begin{equation}\label{compactness_stbf2}
 \varlimsup_{j'} \int_{t_1}^{t_2}\int_{\R^n \times \G } S^{\perp}(\nabla\phi)\cdot h_{j'} \, d\lambda_{j'} \leq \int_{t_1}^{t_2}\int_{\R^n \times \G } S^{\perp}(\nabla\phi)\cdot h \, d\lambda.  
\end{equation}
First, the map $(y,t) \mapsto \int_{\G} S^{\perp} (\nabla \phi) \, d\nu_{(x,t)}(S)$ is measurable and belongs to  $L^2(\mu(t) \otimes dt)$ (as $\nabla\phi$ is bounded and $\nu_{(\cdot,\cdot)}$ are probability measures). Therefore, for all $\eta > 0$, there exists a vector field $g_{\eta} \in C_c^1(\R^n \times [t_1,t_2], \R^n)$ such that 
\begin{equation}\label{compactness_geta}
 \int_{t_1}^{t_2} \int_{\R^n} \big| \int_{\G}S^{\perp}(\nabla\phi)d\nu_{(y,t)}(S) - g_{\eta}(y,t) \big|^2 d\mu(t)(y)dt < \eta^2.
\end{equation}
Now we compute as,
\begin{equation}\label{compactness_stbf2_prf3}
\begin{split}
 \int_{t_1}^{t_2}\int_{\R^n \times \G } S^{\perp}(\nabla\phi)\cdot h_{j'} \, d\lambda_{j'}  
 &= \int_{t_1}^{t_2}\int_{\R^n \times \G } \left( S^{\perp}(\nabla\phi) - g_{\eta} \right) \cdot h_{j'} \, d\lambda_{j'}
 \\& - \delta\lambda_{j'}(g_{\eta}) + \delta\lambda(g_{\eta}) 
 \\& + \int_{t_1}^{t_2}\int_{\R^n \times \G } \left( g_{\eta} - S^{\perp}(\nabla\phi) \right) \cdot h \, d\lambda
 \\& +\int_{t_1}^{t_2}\int_{\R^n \times \G } S^{\perp}(\nabla\phi)\cdot h \, d\lambda.
 \end{split}
\end{equation}
For the first term of the RHS and by the varifold convergence,
\begin{equation}\label{compactness_stbf2_prf1}
\begin{split}
& \varlimsup_{j'} \int_{t_1}^{t_2}\int_{\R^n \times \G } \left( S^{\perp}(\nabla\phi) - g_{\eta} \right) \cdot h_{j'} \, d\lambda_{j'} 
\\& \leq \varlimsup_{j'} \left( \int_{t_1}^{t_2}\int_{\R^n \times \G} |S^{\perp}(\nabla\phi) - g_{\eta}|^2 \, d\lambda_{j'} \right)^{\frac12} \left( \int_{t_1}^{t_2}\int_{\R^n} |h_{j'}|^2 \, d\lambda_{j'} \right)^{\frac12}  \leq \eta \left( \mu(0)(\R^n) \right)^{\frac12} \xrightarrow[\eta \rightarrow 0]{} 0
\end{split}
\end{equation}
where we used  \eqref{compactness_geta}. The second term of the RHS tends to $0$ (for every $\eta$) only by varifold convergence, for the last term of the RHS we write similarly to \eqref{compactness_stbf2_prf1}
\begin{equation}\label{compactness_stbf2_prf2}
\begin{split}
 &\int_{t_1}^{t_2}\int_{\R^n \times \G } \left( g_{\eta} - S^{\perp}(\nabla\phi) \right) \cdot h \, d\lambda 
 \\& \leq  \left( \int_{t_1}^{t_2} \int_{\R^n\times \G} |g_{\eta} - S^{\perp}(\nabla \phi)|^2 \, d \lambda \right)^{\frac12} \left( \int_{t_1}^{t_2}\int_{\R^n} |h|^2 \, d\lambda \right)^{\frac12}
 \leq \eta \left( \mu(0)(\R^n) \right)^{\frac12} \xrightarrow[\eta \rightarrow 0]{} 0
\end{split}
\end{equation}
taking the limit in \eqref{compactness_stbf2_prf3} with respect to $j'$, then w.r.t. $\eta$ and using \eqref{compactness_stbf2_prf1}, \eqref{compactness_stbf2_prf2} we obtain \eqref{compactness_stbf2}.
Finally, the proof of step $3$ stems from \eqref{stbrakkeinequalities}, \eqref{compactness_stbf0}, \eqref{compactness_stbf1} and \eqref{compactness_stbf2}. Hence, according to Definition \ref{def:spacetimebf} , $\lambda$ is a spacetime Brakke flow, this finishes the proof of Theorem \ref{thm:weak_compactness}.

\end{proof}
\newpage
\section*{Acknowledgments}
I would like to express my sincere gratitude to my thesis supervisors, Blanche Buet, Gian-Paolo Leonardi, and Simon Masnou, for their insightful suggestions and support throughout this work.

\bibliographystyle{abbrv}
\bibliography{these}
\end{document}